\documentclass[11pt]{amsart}
\usepackage[T1]{fontenc}
\usepackage{amsmath,a4wide}
\usepackage{amssymb}
\usepackage{amscd}
\usepackage{epsfig}
\usepackage[T1]{fontenc}
\usepackage{epsfig}
\usepackage{amsmath}
\usepackage{dsfont}
\usepackage{graphicx}
\usepackage{mathrsfs}
\usepackage{color}
\usepackage{booktabs}
\usepackage{cite}
\usepackage{amsthm}
\usepackage{amsmath}
\usepackage{enumitem}

\usepackage{pgfplots}
\usepgfplotslibrary{groupplots}
\pgfplotsset{compat=1.12}

\usepackage{tikz}
\usetikzlibrary{arrows,intersections}
\usepackage{boldline}
\usepackage{multirow}
\usepackage[margin=1in]{geometry}

\usepackage{varioref}
\usepackage{xr-hyper}
\usepackage{hyperref}
\hypersetup{
	colorlinks,linkcolor=blue,citecolor=blue,filecolor=red,urlcolor=blue}

\newcommand{\eps}{\varepsilon}
\renewcommand{\phi}{\varphi}

\newcommand{\N}{\mathbb{N}}

\newcommand{\R}{\mathbb{R}}

\renewcommand{\P}{\mathbb{P}}
\newcommand{\E}{\mathbb{E}}

\newcommand{\cB}{\mathcal{B}}

\newcommand{\cD}{\mathcal{D}}

\newcommand{\cI}{\mathcal{I}}
\newcommand{\cJ}{\mathcal{J}}

\newcommand{\cP}{\mathcal{P}}

\newcommand{\cS}{\mathcal{S}}

\newcommand{\cW}{\mathcal{W}}

\newcommand{\sL}{\mathscr{L}}

\newcommand{\tomix}{\xrightarrow[]{\tau_{\mathcal M}}}

\renewcommand{\d}{{\rm d}}

\newcommand{\cpl}{{\rm Cpl}}

\newcommand{\Lip}{{\rm Lip}}

\newcommand{\Rd}{{\R^d}}

\def\Cb{{\rm C}_{\rm b}}

\def\Cci{{\rm C}_{\rm c}^\infty}

\def\Lipb{{\rm Lip}_{\rm b}}
\def\one{\mathds{1}}

\DeclareMathOperator{\mesh}{mesh}

\DeclareMathOperator*{\esssup}{ess\,sup}

\newcommand{\X}{\Rd}

\allowdisplaybreaks

\numberwithin{equation}{section}

\theoremstyle{plain}
\newtheorem{theorem}{Theorem}[section]
\newtheorem{proposition}[theorem]{Proposition}
\newtheorem{lemma}[theorem]{Lemma}
\newtheorem{corollary}[theorem]{Corollary}

\theoremstyle{definition}
\newtheorem{assumption}[theorem]{Assumption}

\newtheorem{remark}[theorem]{Remark}
\newtheorem{definition}[theorem]{Definition}
\newtheorem{setting}[theorem]{Setting}

\theoremstyle{remark}
\newtheorem{example}[theorem]{Example}

\title[Scaling limits of multi-period DRO problems]{Scaling limits of multi-period distributionally robust optimization problems}

\author{Max Nendel}
\address{Department of Statistics and Actuarial Science, University of Waterloo}
\email{mnendel@uwaterloo.ca}

\author{Ariel Neufeld}
\address{Division of Mathematical Sciences, Nanyang Technological University}
\email{ariel.neufeld@ntu.edu.sg}

\author{Kyunghyun Park}
\address{Division of Mathematical Sciences, Nanyang Technological University}
\email{kyunghyun.park@ntu.edu.sg}

\author{Alessandro Sgarabottolo}
\address{Department of Mathematics, LMU M\"unchen}
\email{sgarabottolo@math.lmu.de}

\thanks{\textit{MSC2020 Subject Classification:}\ Primary 90C17; 47H20; Secondary 90C30; 90C31;
60J35}
\thanks{\textit{Funding:}\ M.~Nendel is financially supported by the Natural Sciences and Engineering Research Council of
Canada via Discovery Grant no.\ RGPIN-2025-04219.\ A.~Neufeld gratefully acknowledges support by the MOE AcRF Tier 2 Grant MOE-T2EP20222-0013.\ K.~Park and A.~Sgarabottolo gratefully acknowledge support by the Presidential Postdoctoral Fellowship of Nanyang Technological University and by the National Research Foundation of Korea (Grant DOI: RS-2025-02633175).
}

\date{\today}

\begin{document}

\begin{abstract}
    We examine the scaling limit of multi-period distributionally robust optimization (DRO) problems via a semigroup approach.\ Each period involves a worst-case maximization over distributions in a Wasserstein ball around the transition probability of a reference process with radius proportional to the length of the period, and the multi-period DRO problem arises through its sequential composition. We show that the scaling limit of the multi-period DRO, as the length of each period tends to zero, is a strongly continuous monotone semigroup on $\mathrm{C_b}$.\ Furthermore, we show that its infinitesimal generator is equal to the generator associated with the non-robust scaling limit plus an additional perturbation term induced by the Wasserstein uncertainty.\ As an application, we show that when the reference process follows an It\^o process, the viscosity solution of the associated nonlinear PDE coincides with the value of continuous-time robust optimization problems under parametric uncertainty.
\end{abstract}

\maketitle

\section{Introduction}\label{sec:intro}

When implementing an optimization method based on a stochastic model, it is likely that a margin for potential inaccuracies in the corresponding probabilities or parameters exists.\ In the economic literature this phenomenon is often referred to as Knightian uncertainty, see, e.g., \cite{chen2002ambiguity,dow1992uncertainty,gilboa1989maxmin,garlappi2007portfolio}, or model misspecification, see, e.g., \cite{hansen2022, cerreia2025}. Robust optimization, which presents a worst-case approach, accommodates this issue by solving min-max or max-min problems defined over two sets: the set of admissible controls (to be optimized) and the set of plausible probability distributions (representing model uncertainty).\ In this context, there exists a wide range of mathematical formulations, which can be broadly classified depending on whether the uncertainty stems from probabilistic imprecision, i.e., distributional uncertainty, see, e.g., \cite{bayraktar2023nonparametric,blanchard2018multiple-priors,blanchet2019quantifying,gao2023distributionally,mohajerin2018data,blanchet2022distributionally,obloj2021distributionally,neufeld2023markov,yue2022linear}, or from model parameters, see, e.g., \cite{coquet2002filtration,peng1997backward,denis2011function,nutz2013random,peng2007g,peng2008multi,soner2011quasi,neufeld2017nonlinear,neufeld2014measurability,liu2019compactness,neufeld2018robust}.

A primary interest across these frameworks is to obtain a dynamic equation---typically a partial differential equation\;(PDE)---that characterizes the value function of the robust problem and allows the computation of optimal strategies.
In this context, the dynamic programming principle or semigroup property plays a central role.\
The recent works \cite{denk2018kolmnonlinexpect,nendel2021wasserstein,denk2020semigroup,nendel2021upper,criens2025stochastic,criens2025representation} have extended 
classical semigroup theory to dynamic optimization problems, see also \cite[Chapter II]{fleming2006controlled} for a broad discussion on the use of semigroups in the context of viscosity solutions and optimal control.\ Moreover, in a series of works \cite{BDKN22,BK23,BK22,BKN23}, Chernoff-type approximation results for nonlinear expectations have been developed, which allow to recover dynamic consistency starting from families of operators that need not be dynamically consistent.
The key idea of a Chernoff approximation is to start from a family of one-step operators, e.g., $(I(t))_{t \ge 0}$ with $I(t) \colon \Cb \to \Cb$, which satisfy suitable properties, and show that these properties are preserved under iterations of the type
\[
I(t_1-t_0)\circ \cdots \circ I(t_k - t_{k-1}), \quad \text{for} \quad 0 = t_0  < \dots < t_k=t,
\]
where $\{t_0, \dots, t_k\}$ is a partition of a finite time interval $[0,t]$.\
The dynamic consistency that is naturally obtained on the multi-period level can then be transferred to continuous time when taking the limit over increasingly finer partitions of finite time intervals. While such results are obtained under rather abstract conditions and can be applied to a variety of frameworks, they require convexity of the one-step operators $(I(t))_{t \ge 0}$. Hence, they are not suitable for addressing robust optimization problems, where, in general, the additional infimum breaks convexity. Some exceptions arise when the control enters in the one-step operators in a convenient way, preserving the convexity of the operators themselves, see, e.g., \cite{blessing2024discrete}.\ However, a direct application of the previous results in a general robust optimization framework is not possible.

The papers\ \cite{bartl2021limits,nendel2021wasserstein}, which consider transition semigroups for classes of time-homogeneous Markov processes under distributional uncertainty, form the starting point of this article.\
In these works, the authors study a sequential composition of nonlinear expectation operators under distributional uncertainty, where the nonlinear expectation operators are modeled as a perturbation of the transition probabilities of the Markov process within a proximity in Wasserstein distance.  In this setting, it is shown that, 
as the time horizon is divided into progressively finer periods with a proportional scaling of the level of uncertainty,
the sequential composition converges to a semigroup in the mixed topology, see \eqref{eq:mix_top} in Section \ref{sec:notation}.\ Furthermore, by examining the infinitesimal generator of the semigroup, it is shown that the semigroup gives rise to viscosity solutions to a nonlinear PDE.\

    In this article, we extend the results of \cite{nendel2021wasserstein} by proposing a semigroup approach for a class of distributionally robust optimization (DRO) problems, thus providing an analytic bridge between Wasserstein distributionally robust optimization problems, see, e.g., \cite{blanchet2019quantifying,gao2023distributionally,mohajerin2018data,blanchet2022distributionally,obloj2021distributionally,neufeld2023markov,yue2022linear}, and continuous-time robust optimization problems, see, e.g., \cite{hernandez2007control,schied2006risk, maccheroni2006ambiguity,schied2007optimal,talay2002worst,park2022robust,park2023robust}.\ While certain proof techniques in this article resemble those used in \cite{bartl2021limits,nendel2021wasserstein} for nonlinear expectations involving merely a supremum (or infimum) over the uncertainty set, the min-max structure of the robust optimization introduces major technical challenges.\ At the core of our analysis lies an explicit description of the nonlinear infinitesimal generator, which is closely tied to the sensitivity of DRO problems, a topic that has recently gained a significant amount of attention.\ We refer to \cite{bartl2021sensitivity,gao2024wasserstein,gotoh2025distributionally,nendel2022parametric} for sensitivity analyses for Wasserstein distributionally robust optimization problems in one period, \cite{jiang2024sensitivity,sauldubois2024first,bartl2023sensitivity1} in a multi-period setting, and \cite{herrmann2017model,herrmann2017hedging,bartl2023sensitivity,bartl2024numerical,jiang2024sensitivity, touzi2025sensitivity} in continuous time.\ The works \cite{bartl2023sensitivity,bartl2024numerical} investigate the sensitivity of robust optimization problems under drift and volatility uncertainty, while \cite{touzi2025sensitivity} studies local sensitivity in non-Markovian DRO control and stopping problems under drift perturbations. In the latter work, explicit first-order formulas are derived via BSDEs and reflected BSDEs. These contributions are complementary to the present paper in that they analyze infinitesimal robustness of already formulated continuous-time robust optimization problems, whereas we start from discrete-time multi-period Wasserstein DRO problems and identify their continuous-time scaling limit through a nonlinear semigroup approach.\ The paper \cite{jiang2024sensitivity} studies sensitivity for continuous-time causal distributionally robust optimization problems.\ The sensitivities therein are expressed using optional projections of Malliavin derivatives, and are obtained as limits of their discrete-time counterparts.

Multi-period DRO arises in a wide range of fields.\ We refer, e.g., to \cite{neufeld2023markov,neufeld2026nonconcave,blanchard2018multiple,nutz2016utility,bayraktar2023nonparametric,bartl2019exponential} for robust utility maximization in finance, to \cite{yu2022multistage,pichler2021mathematical,gao2024data} for robust stochastic programming in optimization, and to \cite{yang2023distributionally,shin2022distributionally,xin2021time} for robust inventory and logistics planning in engineering and management. 

Multi-period DRO formulations are typically posed on a fixed discrete time grid, but it is natural to ask what happens as the time discretization is refined and the length of each period tends to zero.\ The scaling limit of multi-period DRO problems derived in this paper establishes a connection between discrete-time multi-period DRO problems and continuous-time robust optimization models, thereby clarifying the continuous-time robust dynamics that emerge from the accumulation of distributional uncertainty over progressively shorter time horizons.

We proceed with a rigorous description of our results.\ We consider semigroups on the space $\Cb$ of all bounded continuous functions $\X\to \mathbb{R}$.\ For $p\geq 1$, let $\cP_p$ be the set of all probability measures on the Borel $\sigma$-algebra $\cB(\X)$ with finite moment of order $p$.\ Throughout, we fix a nonempty set of actions $A$. For each action $a\in A$, let
$(\mu_t^a)_{t\ge 0} \subset \cP_p$ be a family of probability measures and let $(\psi_t^a)_{t\ge 0}$ be a family of continuous maps $\X \to \X$. For any  $a \in A$ and $t\geq 0$, we define a linear operator $T^a(t)$ on $\Cb$ by setting
\begin{align}\label{eq:Intro1}
    \begin{aligned}
        \big( T^a(t)f\big)(x)
        :=
        \int_\X f(\psi_t^a(x) + y_t)\,\mu_t^a(\d y_t) \quad\text{for all } x\in \X \text{ and } f\in \Cb.
    \end{aligned}
\end{align}
Moreover, we assume that, for each $a\in A$, the family $(T^a(t))_{t\geq 0}$ defined in\;\eqref{eq:Intro1} is the transition semigroup of a time-homogeneous controlled Markov process, and we refer to it as the \textit{reference semigroup} associated with the action $a$. Natural examples include Brownian motions with drift and Ornstein-Uhlenbeck processes, see Example~\ref{lem:examples}.
    
To account for distributional uncertainty in the reference measure $\mu^a$, we consider the set 
\begin{align}\label{eq:Intro2}
    \mathcal{B}_t^a(m):=\{\nu\in \cP_p\,|\,\cW_p(\mu_t^a, \nu) \le t m \}\quad \text{for all }t\geq 0\text{ and } a\in A,
\end{align}
where $\cW_p$ denotes the $p$-Wasserstein distance and $m\geq 0$ can be seen as a confidence parameter describing the degree of uncertainty.

We then define a family of one-period operators $(I(t))_{t\geq 0}$ on $\Cb$ by 
\begin{align}\label{eq:Intro3}
    (I(t)f)(x):= \inf_{a\in A} \sup_{\nu \in \mathcal{B}^a_t(m)}\int_\X f(\psi_t^a(x)+z)\nu(dz)\quad \text{for all }t\geq 0,\; f\in \Cb,\text{ and } x\in \X.
\end{align}
Thus, $(I(t)f)(x)$ represents the value of a distributionally robust optimization problem for the uncertainty set $\mathcal{B}_t^a(m)$ in \eqref{eq:Intro2}.

To obtain a dynamically consistent family from $(I(t))_{t\geq 0}$, we consider the set of partitions $\operatorname{P}_t:=\{\{t_0, \dots, t_k\}\,|\, 0=t_0 < \cdots < t_k = t,\;k\in\mathbb{N}\}$ for $t\geq 0$. We then consider the sequential composition operator $\mathcal{I}(\pi)$ over a partition $\pi=\{t_0,\ldots,t_k\}\in\operatorname{P}_t$ for $t\geq 0$, given by 
\begin{align}\label{eq:Intro4}
    \big(\mathcal{I}(\pi)f\big)(x) := \big(I(t_1 - t_0) \cdots I(t_k - t_{k-1}) f\big)(x)\quad \text{for all } f\in \Cb \text{ and } x\in \X,
\end{align}
which represents the multi-period version of the distributionally robust optimization problem given in \eqref{eq:Intro3}. Here and throughout, for $s,t \ge 0$, we use the notation $I(s)I(t) := I(s) \circ I(t)$ for the composition of the operators $I(s)$ and $I(t)$, and similarly for the composition of all operators in the paper.

With these notations and definitions in place, Theorem \ref{thm:semigroup} shows that if the reference components $(\mu_t^a,\psi_t^a)_{t\geq0}$ appearing in \eqref{eq:Intro1} satisfy certain integrability and regularity conditions, see Assumption \ref{ass:general}, then the family of operators $(\mathcal{S}(t))_{t\ge 0}$, defined by 
\begin{align}\label{eq:Intro5}
    (\mathcal{S}(t)f)(x) := \inf_{\pi \in {\rm P}_t} (\mathcal{I}(\pi)f)(x)\quad \text{for all }t\geq0, \; f\in \Cb,\text{ and }  x\in \X,
\end{align}
is a strongly continuous monotone semigroup on $\Cb$, see Definition \ref{def:semigroup}.\ The additional infimum over partitions of the interval $[0,t]$ together with the infimum over $a\in A$ in \eqref{eq:Intro3} allows to optimally choose both the times where one acts on the system and a corresponding action at that time.\ Proposition \ref{lem:I.on.Cb} implies that the family of operators $\big(\mathcal I(\pi)\big)_{\pi\in {\rm P}_t}$ is nonincreasing over refining partitions, so that the infimum in \eqref{eq:Intro5} can be replaced by a limit over refining, e.g., dyadic partitions.\ This result can be seen as an analogue of the scaling limit derived in \cite{nendel2021wasserstein} with an additional optimization: as the partition in $\operatorname{P}_t$ becomes increasingly finer, the family of sequential composition operators in \eqref{eq:Intro4} converges to a semigroup $(\mathcal{S}(t))_{t\geq 0}$ in the mixed topology, see Section \ref{sec:notation} for its definition.\ We point out that the mixed topology is a natural choice to obtain strong continuity of transition semigroups on $\Cb$ since the supremum norm is often too strong in order to have a densely defined generator.\ For example, the transition semigroup of an Ornstein-Uhlenbeck process is known not to generate a semigroup w.r.t.\ the uniform topology, not even when passing to the space of bounded uniformly continuous functions, while it does for the mixed topology, see \cite{kocan} and the discussion in \cite{Goldys2022mixed}.\ Moreover, there are prominent non-existence results for strongly continuous semigroups w.r.t.\ the supremum norm, cf.\ \cite{lotz} for the case where the underlying Banach space is $L^\infty$.

Furthermore, given the semigroup $(\mathcal{S}(t))_{t\geq 0}$ in \eqref{eq:Intro5}, Theorem \ref{thm:semigroup.generator} shows that, under the same assumptions on $(\mu_t^a,\psi_t^a)_{t\geq0}$, the infinitesimal generator $\mathscr{L}$ of $(\mathcal{S}(t))_{t\geq 0}$, see Definition \ref{def:semigroup_infinitesimal}, satisfies
\begin{align}\label{eq:Intro6}
    (\mathscr{L}f)(x)=\inf_{a \in A}(\mathscr{L}^a f)(x) + m \|\nabla f(x)\|\quad \text{for all }x \in \X
\end{align}
 if the function $f$ is sufficiently regular, where, for each $a \in A$, $\mathscr{L}^a$ is the infinitesimal generator of the reference semigroup $T^a$ in \eqref{eq:Intro1}. In other words, $\mathscr{L}f$ equals the Bellman operator $\inf_{a \in A}\mathscr{L}^a f$ corresponding to an optimal control problem over the action set $A$ plus an additional perturbation term $m\|\nabla f(\cdot)\|$, induced by the Wasserstein uncertainty in \eqref{eq:Intro2}.

Lastly, as an application of our main results, Theorem \ref{cor:viscosity} shows that the map $v\colon [0,\infty)\times \X \to \mathbb{R}$, defined by $v(t,x):=(\mathcal{S}(t)u_0)(x)$ for $(t,x)\in [0,\infty)\times \X$, is a viscosity solution to the PDE with initial condition $v(0,x)=u_0(x)$, see \eqref{eq:viscosity} and Definition \ref{def:D.M.viscosity}, which is  governed by the nonlinear integro-differential operator $\mathscr{L}$.\ We point out that the infinitesimal generator is of (additively separated) inf-sup-type, so that the aforementioned PDE corresponds to the Isaacs equation related to an abstract inf-sup optimal control problem.\ Theorem \ref{cor:viscosity} then allows us to show that, for specific choices of the reference components $(\mu^a,\psi^a)_{a\in A}$, the viscosity solution $v$ coincides with values of a $g$-expectation, see Corollary \ref{cor:g_exp}, or with the value function of a continuous-time robust optimization problem under drift uncertainty, see Corollary \ref{cor:SDG}.\ This resembles a collapse of nonparametric uncertainty to parametric uncertainty in the scaling limit, which has also been observed in \cite{bartl2021limits,nendel2021wasserstein} without an additional optimization.

\vspace{0.5em}
The rest of the paper is structured as follows. In Section \ref{sec:main}, we introduce the setup and state the first two main results, Theorem \ref{thm:semigroup} and Theorem \ref{thm:semigroup.generator}, concerning continuity properties, dynamic consistency, and the infinitesimal behavior of the scaling limit $\cS$ of the multi-period DRO problems.\ Section \ref{sec:apply} contains the third main result, Theorem \ref{cor:viscosity}, which shows that the semigroup~$\cS$ gives rise to viscosity solutions to a possibly nonlocal Hamilton-Jacobi-Bellman-Isaacs (HJBI) equation, and applications in the context of $g$-expectations and robust optimization with drift uncertainty.\ Section \ref{sec:best-case} discusses the related `best-case' optimization, which plays a crucial role in the proofs of the main results.\ Sections \ref{sec:proofs.single.period}, \ref{sec:proofs.DRO.multi}, and \ref{sec:proofs.sec.appl} contain the proofs of Sections \ref{sec:DRO}, \ref{sec:multi.DRO}, and \ref{sec:apply}, respectively.\ Appendix \ref{sec:apdx} contains a series of auxiliary statements.

    \section{Construction of the scaling limit of multi-period DROs}\label{sec:main}
        
	\subsection{Notation and preliminaries}\label{sec:notation} Throughout this paper, let $\|\cdot\|$ and $\langle \cdot, \cdot \rangle$ be the Euclidean norm and inner product on~$\X$, respectively.\ If a subset $K \subset \X$ is compact, we write $K \Subset \X$. 

	{We denote by $\Cb=\Cb(\X)$ the set of all continuous functions $f \colon \X \to \R$ with
    \[
     \|f\|_\infty:=\sup_{x\in \mathbb{R}^d}|f(x)|<\infty.
    \]
    We endow the space $\Cb$ with the {\it mixed topology} $\tau_{\mathcal M}$ between the topology induced by the supremum norm $\|\cdot\|_\infty$ and the topology $\tau_{\mathcal C}$ of uniform convergence on compacts, i.e., the strongest locally convex topology on $\Cb$ that coincides with $\tau_{\mathcal C}$ on $\|\cdot\|_\infty$-bounded subsets of $\Cb$, cf.\ \cite{haydon, wiweger}. We point out that, in this particular setup, the mixed topology $\tau_{\mathcal M}$ is the Mackey topology of the dual pair $(\Cb,{\rm ca})$, where ${\rm ca}$ denotes the space of all countably additive signed Borel measures on $\X$ of finite variation. Moreover, $\tau_{\mathcal M}$ belongs to the class of \textit{strict topologies}, cf.\ \cite{buck,wheeler1, sentilles}.\ We also refer to \cite{kunze,Goldys2022mixed,kocan} for surveys on semigroups in mixed or strict topologies and their relation to Markov processes. 

    While the explicit definition of the mixed topology is rather involved, convergence of sequences in the mixed topology can be characterized as follows.\ A sequence $(f_n)_{n \in \N} \subset \Cb$ converges to $f \in \Cb$ in the mixed topology if and only if
	\begin{align} \label{eq:mix_top}
		\sup_{n \in \N} \|f_n\|_\infty < \infty \quad\text{and}\quad \lim_{n \to \infty} \|f_n - f\|_{\infty, K} = 0 \quad \text{for every } K \Subset \X,    
	\end{align}
 	where $\|f\|_{\infty, K} := \sup_{x \in K}|f(x)|$ for $f\in \Cb$ and $K\Subset \X$.\ In this case, we write $f_n\tomix f$.
    
    We point out that the mixed topology is not metrizable and continuity is not equivalent to sequential continuity.\ 
    However, for convex and concave monotone operators defined on $\Cb$ and taking values in a locally convex vector lattice, continuity in the mixed topology is equivalent to sequential continuity, cf.\ \cite[Corollary 2.9]{Nendel2022lsc}, so that, for our purposes, it is enough to restrict our attention to the convergence of sequences in the mixed topology and sequential continuity of operators defined on $\Cb$.
    
    In our analysis, we also consider the subspace $\Lipb$ of all bounded Lipschitz continuous functions $f \colon \X \to \R$. For $f\in \Lipb$, we write $\|f\|_\Lip$ for the (smallest) Lipschitz constant of $f$, i.e.,
    \[
        \|f\|_{\Lip}:=\sup_{\substack{{x,y\in \X}\\{x\neq y}}} \frac{|f(x)-f(y)|}{\|x-y\|}.
    \]
    Moreover, for $k\in \N$, $\Cb^k$ denotes the space of all $k$-times continuously differentiable functions $f\in \Cb$ with bounded derivatives up to order $k$ and ${\rm C}_{\rm c}^k$ denotes the space of all $f\in \Cb^k$ with compact support $\operatorname{supp}(f)$.\ As usual, we set $\Cci:=\bigcap_{k\in \N} {\rm C}_{\rm c}^k$, and we use the notation $\nabla f$ for the gradient of $f$ and $\nabla^2 f$ for its Hessian.

	Finally, throughout this article, we fix some $p \in (1, \infty)$ and write $\cP_p:=\cP_p(\X)$ for the set of probability measures $\mu$ on the Borel $\sigma$-algebra $\cB(\X)$ with finite moment of order $p$, i.e., $\int_\X \|y\|^p\,\mu(\d y) < \infty$.
	On $\cP_p$, we consider the $p$-Wasserstein distance $\cW_p$, defined by
	\[
		\cW_p(\mu, \nu):= \bigg(\inf_{\gamma\in\cpl(\mu, \nu)} \int_{\X \times \X} \|z - y\|^p \,\gamma(\d y, \d z)\bigg)^{1/p},
	\]
	where $\cpl(\mu, \nu)$ is the set of all couplings between $\mu$ and $\nu$, i.e., probability measures on the product $\sigma$-algebra of $\cB(\X)$ on $\X \times \X$ with first marginal $\mu$ and second marginal $\nu$.\ We refer to \cite{ambrosio2021lectures,villani2008optimal} for a detailed discussion of the Wasserstein distance in the broader context of optimal transport.

	\subsection{The controlled reference Markov process and its transition semigroup}\label{sec:CMP.semigroup}
	Throughout, we consider a fixed nonempty \textit{action set} $A$.\ Following the framework in \cite{nendel2021wasserstein}, we consider a time-homogeneous controlled Markov process along with its transition semigroup, which is assumed to be described by the following two components. 
        \begin{setting}\label{dfn:refer_msr_ft}
        For any $a\in A$, 
        \begin{itemize}
            \item [(i)] $\mu^a:=(\mu_t^a)_{t\ge 0} \subset \cP_p$ is a family of probability measures with finite $p$-moment,
            \item [(ii)] $\psi^a:=(\psi_t^a)_{t\ge 0}$ is a family of continuous maps $ \X \to \X$.
        \end{itemize}
        We assume that the families $\mu^a$ and $\psi^a$ satisfy the \textit{Chapman-Kolmogorov equations}
            \begin{align} \label{eq:chap-kolm}
            	\quad \int_\X f(\psi_{s+t}^a(x) + y_{s+t})\,\mu_{s+t}^a(\d y_{s+t}) = \int_\X \int_\X f\Big( \psi_s^a\big(\psi_t^a(x) + y_t \big) + y_s \Big)\,\mu_s^a(\d y_s)\,\mu_t^a(\d y_t)
            \end{align}
            for all $s,t\geq 0$ and $f\in \Cb$.
        \end{setting}
        
        For each $a\in A$, the families $\mu^a$ and $\psi^a$ correspond to a time-homogeneous Markov process $X^a=(X^{\cdot,a}_t)_{t\geq 0}$ of the form
        \begin{align} \label{eq:ref_proc}
		X_t^{x, a} := \psi_t^a(x) + Y_t^a \quad \text{with}\quad Y_t^a \sim \mu_t^a\quad\text{for all }x\in \X\text{ and } t\geq0,
	\end{align}   
    whose transition kernels $p_t^a\colon \X\to \cP_p,\; x \mapsto p_t^a(x;\cdot)$ are given by
        \begin{align}\label{eq:ref_trans_prob}
		p_t^a(x;B) := \mu_t^a\big( \{ y \in \X \,|\, \psi_t^a(x)+y \in B\} \big)
	\end{align}
     for all $t\geq 0$, $x\in \X$, and $B \in \cB(\X)$.
	
	For any  $a \in A$ and $t\geq 0$, we define a linear operator $T^a(t)$ on $\Cb$ via
	\begin{align} \label{eq:ref_sg}
            \begin{aligned}
		\big( T^a(t)f\big)(x)
            :=\int_{\X}f( x_t)\, p_t^a(x;\d x_t)
            =
		\int_\X f(\psi_t^a(x) + y_t)\,\mu_t^a(\d y_t) 
            \end{aligned}
	\end{align}
    for all $f\in \Cb$ and $x\in \X$, where the last equality follows from the particular shape of the transition kernel $p_t^a$ in \eqref{eq:ref_trans_prob}.
        
    By \eqref{eq:chap-kolm},  for each $a\in A$, the family $T^a:=(T^a(t))_{t\geq 0}$ in \eqref{eq:ref_sg} satisfies the semigroup property
	\[
	T^a(s+t)f = T^a(t)T^a(s)f\quad \text{for all }s,t \ge 0\text{ and }f \in \Cb.
	\]
    That is, $T^a$ is the transition semigroup of the time-homogeneous Markov process $(X_t^{\cdot,a})_{t\geq 0}$ in~\eqref{eq:ref_proc}, and we refer to it as the \textit{reference semigroup} associated to the action $a\in A$.

    For our analysis, we need to impose several conditions for the behavior of the reference semigroups $(T^a)_{a\in A}$, more precisely, of its components $(\mu^a)_{a\in A}$ and $(\psi^a)_{a\in A}$, which we collect in the following assumption.
    
	\begin{assumption}\label{ass:general} The following conditions hold:
		\begin{enumerate}[leftmargin=3.em, label=(\roman*)]
            \item 
			For all $(x,a) \in \X\times A$, $\psi_0^a(x) = x$. Moreover, there exists a constant $C> 0$ such that, for all $x \in \X$,
            \[
            \sup_{t > 0} \sup_{a \in A} \frac{\|\psi_t^a(x) - x\|}{t} \le C(1 + \|x\|).
            \]
            \item
			There exists a constant $c \ge 0$ such that, for every $t \ge 0$, $(\psi_t^a)_{a\in {A}}$ are Lipschitz continuous with their Lipschitz constants uniformly bounded by $e^{ct}$, i.e., for all $t\geq0$ and $x_1,x_2\in \X$,
			\[
				\sup_{a \in A} \|\psi_t^a(x_1) - \psi_t^a(x_2)\| \le
				e^{ct} \|x_1 - x_2\|.
			\]
			\item
			For every $\chi \in \Cci$, it holds that
			\[
			\limsup_{t \downarrow 0} \sup_{a \in A} \bigg| \int_\X \frac{\chi(y) - \chi(0)}{t}\,\mu_t^a(\d y) \bigg| < \infty.
			\]
			\item 
			For all $a\in A$, $\mu_0^a = \delta_0$ and $\lim_{t \downarrow 0} \sup_{a \in A} \int_\X \|y\|^p\,\mu_t^a(\d y) = 0.$ Moreover, for every $\eps > 0$, there exists some $M >0$ such that
			\[
			\sup_{a \in A} \sup_{t > 0} \frac{\mu_t^a\big(\big\{y \in \X \,\big|\, \|y\| >  M\big\}\big)}{t} < \eps.
			\]
		\end{enumerate}
	\end{assumption}

    Before we start our analysis, we provide two examples for $(\mu^a)_{a\in A}$ and $(\psi^a)_{a\in A}$ satisfying Assumption \ref{ass:general} and the Chapman-Kolmogorov equation \eqref{eq:chap-kolm}.
    
        \begin{example}\label{lem:examples} Let $(\Omega,\mathcal{F},(\mathcal{F}_t)_{t\geq 0},\mathbb{P})$ be a 
        filtered probability space and let $(W_t)_{t\geq 0}$ be a $d$-dimensional Brownian motion on that filtered probability space. Moreover, let $A$ be a nonempty set.
           \begin{enumerate}
               \item [a)] (Brownian motion with drift) For all $t\geq 0$, $a\in A$, and $x\in \X$, define
                \begin{align}\label{eq:eg_psi}
                &\psi_t^a(x):=x+b(a)t, \\
                & Y_t^a=  \sigma(a)W_t,
                   \quad  Y_0^a=0,\label{eq:ito_diffusion}
                \end{align}
                with bounded functions $b\colon A\to \X$ and $\sigma\colon A\to \mathbb{R}^{d\times d}$.
               \item [b)] (Ornstein-Uhlenbeck process) For every $t\geq 0$, $a\in A$, and $x\in \X$, define
                \begin{align}\label{eq:eg_psi_2}
                &\psi_t^a(x):=e^{-\theta(a)t}x+\int_0^te^{-\theta(a)s}\kappa(a)\,\d s,  \\
                &Y_t^a = \int_0^t e^{-\theta(a)s}\sigma(a)\,{\rm d}W_s,
                \label{eq:OU_process}
                \end{align}
               where $\theta \colon A \to \R^{d \times d}$, $\kappa\colon A \to \mathbb{R}^d$, and $\sigma\colon A\to \mathbb{R}^{d\times d}$ are bounded functions and $\theta(a)$ is symmetric and positive semi-definite for all $a\in A$.
           \end{enumerate}
            Then, both previously described cases satisfy Assumption \ref{ass:general} and the Chapman-Kolmogorov equations \eqref{eq:chap-kolm}.
            
            Indeed,
             Assumption \ref{ass:general} (i) and (ii) immediately follow from the explicit form of $(\psi^a)_{a\in A}$ together with the boundedness of the functions $b$, $\kappa$, and $\theta$ and the fact that $\theta(a)$ is symmetric and positive semi-definite for all $a\in A$.\ Since both cases lead to Markov processes $(X_t^{\cdot, a})_{t\geq0}$ for each $a\in A$, the Chapman-Kolmogorov equations \eqref{eq:chap-kolm} are satisfied.

             For Assumption \ref{ass:general} (iii) and (iv), observe that
             \begin{align}
             \label{eq.normal.example1} &Y_t^a\sim \mathcal N\big(0,\sigma(a) \sigma(a)^\top t\big)\quad \text{or}\\
             &Y_t^a\sim \mathcal N\bigg(0,\int_0^t e^{-\theta(a)(t-s)}\sigma(a) \sigma(a)^\top\big(e^{-\theta(a)(t-s)}\big)^\top\,{\rm d}s\bigg) \label{eq.normal.example2}
             \end{align}
             in a) or b), respectively, where $\mathcal N$ refers to a $d$-dimensional normal distribution. Recall that, for a normally distributed random vector $Y\sim \mu:=\mathcal N(0,\sigma\sigma^\top)$ with $\sigma\in \R^{d\times d}$, $\chi\in \Cci$ and $t\geq0$, by Taylor's theorem,
             \[
                \int_\X \frac{\chi(y)-\chi(0)}{t}\,\mu(\d y)\leq \frac{\|\sigma\|^2}{2t}\|\nabla^2\chi\|_\infty
             \]
             and, by Markov's inequality,
            \[
           \frac{\mu\big(\big\{y \in \X \,\big|\, \|y\| >  M\big\}\big)}{t} \leq \frac{\|\sigma\|^2}{M^2t}.
			\]
            Hence, using the boundedness of $\sigma$ and the fact that $\theta(a)$ is symmetric and positive semi-definite for all $a \in A$, properties (iii) and (iv) follow in both cases, since $\mu_0^a=\delta_0$ for all $a\in A$ and $\lim_{t \downarrow 0} \sup_{a \in A} \int_\X \|y\|^p\,\mu_t^a(\d y) = 0$ due to the fact that $\mu_t^a$ is a normal distribution of the form \eqref{eq.normal.example1} or \eqref{eq.normal.example2}.
        \end{example}

	\subsection{Single-period distributionally robust optimization}\label{sec:DRO}
	In this section, we introduce a family of operators that represent a class of single-period distributionally robust optimization (DRO) problems.
    
    Recall the family $(\mu^a)_{a\in A}$ given in Setting \ref{dfn:refer_msr_ft} (i).\ Throughout, we fix a radius $m\geq0$, which corresponds to a \textit{degree of uncertainty} that is scaled proportionally to time. Then, for all $a\in A$ and $t\geq 0$, we denote by
	\begin{align}\label{eq:uncertain.set}
	\mathcal{B}_t^a(m):=\big\{\nu\in \cP_p\,\big|\,\cW_p(\mu_t^a, \nu) \le t m \big\}
	\end{align}
	the set of all probability measures that lie within a $t m$-radius of the reference distribution $\mu_t^a\in \mathcal{P}_p$ w.r.t.\ the $p$-Wasserstein distance $\mathcal{W}_p$. 
	
	We define a family of operators $I:=(I(t))_{t \ge 0}$  on~$\Cb$ by setting
	\begin{align}\label{eq:nonlinear_I}
		\quad
		\begin{aligned}
		&\big( I(t)f \big)(x) := \inf_{a \in A} \sup_{\nu\in\mathcal{B}^a_t(m)} \int_\X f\big( \psi_t^a(x) + z\big)\,\nu(\d z) \\
		\end{aligned}
	\end{align}
    for all $t\geq0$, $f\in \Cb$, and $x\in \X$,
	so that $I(t)f$ corresponds to a static DRO problem for the ambiguity sets $(\mathcal{B}_t^a(m))_{a\in A}$ given in \eqref{eq:uncertain.set}.\ Moreover, we consider the family of operators $T:=(T(t))_{t\geq 0}$, given by
      \[
        \big(T(t)f\big)(x):=\inf_{a \in A} \big(T^a(t)f\big)(x)\quad \text{for all }t\geq 0,\,f\in \Cb,\text{ and }x\in \X.
    \]
    If $m=0$, we have $I(t)=T(t)$ for all $t\geq 0$, i.e., the family $(I(t))_{t\geq 0}$ corresponds to a minimization of the reference semigroups $(T^a(t))_{t\geq 0}$ given in \eqref{eq:ref_sg} over $a\in A$.

    The following proposition forms the starting point of our theoretical analysis. Its proof can be found in Section \ref{sec:proofs.single.period}. For similar statements without the additional minimization, we refer to \cite{bartl2021limits,nendel2021wasserstein}.

    \begin{proposition} \label{lem:I.on.Cb}
		Suppose that Assumption \ref{ass:general} holds.
    \begin{enumerate}[label=(\roman*)]
        \item For every $t \ge 0$, the operator $I(t) \colon \Cb \to \Cb$ is well-defined, monotone, and a contraction, i.e., $I(t)f\geq I(t)g$ for all $f,g\in \Cb$ with $f\geq g$ and
        \[
         \|I(t)f - I(t)g\|_\infty \le \|f - g\|_\infty\quad\text{for all }f,g\in \Cb.
        \]
        \item For all $t \ge 0$ and $f\in \Lipb$, we have $\|I(t)f\|_{\Lip}\leq e^{ct}\|f\|_{\Lip}$, where $c \ge 0$ is defined in Assumption~\ref{ass:general}~(ii).
        \item For all $s,t \ge 0$, $f\in \Cb$ and $x \in \X$,
		\[
		\big( I(s)I(t)f \big)(x) \le \big(I(s+t)f\big) (x).
		\]
		\end{enumerate}
        In particular, choosing $m=0$, these properties also hold for the family $T$ instead of $I$.
	\end{proposition}
    
    We now turn our focus on the asymptotic relation between the families $I$ and $T$ as $t\to 0$, which can be understood as a sensitivity w.r.t.\ the degree of uncertainty, cf.\ \cite{bartl2021sensitivity,nendel2022parametric} and \cite{bartl2023sensitivity,bartl2023sensitivity1} in a dynamic context.\ 

    \begin{remark}\label{rem.gamma}
    Observe that, by \cite{villani2008optimal}, for all $f\in \Cb^1$, $t >0$, $x \in \X$, $a\in A$, and $\eps>0$, there exists $\nu_t^{a,x} \in \cB_t^a(m)$ and an optimal coupling $\gamma_t^{a, x} \in \cP_p(\X \times \X)$ between $\mu_t^a$ and $\nu_t^{a,x}$ such that
    \begin{align}
    \notag \frac{\big(I^a(t)f\big)(x)- \big(T^a(t)f\big)(x)}{t}&\leq  \frac1t\bigg(\int_\X f\big(\psi_t^a(x) + z\big)\,\nu_t^{a, x}(\d z)- \big(T^a(t)f\big)(x)\bigg)+\eps\\
    &=\frac1t\int_{\X\times \X}\Big(f\big(\psi_t^a(x)+z\big)-f\big(\psi_t^a(x)+y\big)\Big)\,\gamma_t^{a,x}(\d y,\d z) +\eps\notag \\
    &\leq \|\nabla f\|_\infty \frac1t\int_{\X\times \X}\|z-y\|\,\gamma_t^{a,x}(\d y,\d z) +\eps\notag \\
    & \le \|\nabla f\|_\infty \frac{\cW_p(\mu_t^a,\nu_t^{a,x})}{t}+\eps \leq m\|\nabla f\|_\infty+\eps \notag,
    \end{align}
    where 
    \begin{equation}\label{eq:def.Ia}
    \big(I^a(t)f\big)(x):= \sup_{\nu\in \cB_t^a(m)} \int_{\X} f\big(\psi_t^a(x)+z\big)\,\nu(\d z)
    \end{equation}
    for all $a\in A$, $t\geq0$, $f\in \Cb$, and $x\in \X$.\ Letting $\eps\to 0$, we thus obtain the following a priori estimate
    \begin{equation}
     \frac{\big(I^a(t)f\big)(x)- \big(T^a(t)f\big)(x)}{t} \leq m\|\nabla f\|_\infty, \label{eq.le.gen.1}
    \end{equation}
     which is further strengthened in the following proposition.
     \end{remark}
    
    \begin{proposition}\label{lem:generator}
		Suppose that Assumption \ref{ass:general} holds. 
            For every $m\geq 0$ and every $f \in \Cb^1$, 
		\begin{align}\label{eq:sensitivity}
		      \lim_{t \downarrow 0} \frac{I(t)f - T(t)f}{t} = m \|\nabla f(\,\cdot\,)\|,
		\end{align}
		where the limit is to be understood in the mixed topology.\footnote{Here and throughout $\|\nabla f(\,\cdot\,)\|$ refers to the map $\X \to \R$, $x \mapsto \|\nabla f(x)\|$.}
	\end{proposition} 

    The proof of Proposition \ref{lem:generator} is contained in Section \ref{sec:proofs.single.period}.

	\subsection{Multi-period DRO and its scaling limit}\label{sec:multi.DRO}
	In this section, we consider a multi-period version of the DRO problem described in \eqref{eq:nonlinear_I}. 
    By `multi-period', we mean that the decision maker {\it iteratively} optimizes their worst-case objective over a finite time horizon.\ We then consider the scaling limit of this multi-period DRO problem as the number of decisions within the finite time horizon tends to infinity.
    
    To that end, we consider the set
    \[
    \operatorname{P}:=\{\pi\subset [0,\infty)\,|\, 0\in \pi,\, \# \pi<\infty\}
    \]
    of all finite partitions of closed intervals on the positive half line that contain zero. Observe that, by definition, every $\pi\in \operatorname{P}$ is either the set containing only zero or of the form $\pi=\{t_0,\ldots, t_k\}$ with $0=t_0<\ldots <t_k$ and $k \in \N$.
    In the latter case, we define $$\mesh \pi:= \max_{i=1,\ldots, k} (t_i-t_{i-1})$$ as the \textit{maximal mesh size} and, for the sake of completeness, we set $\mesh \{0\}:=0$.
    
    Then, for $\pi = \{ t_0, \dots, t_k \} \in \operatorname{P}$ with $0=t_0<\ldots<t_k$ and $k\in \N$, we define 
	\begin{align} \label{eq:def_iter}
		\mathcal{I}(\pi)f := \big(I(t_1 - t_0) \cdots I(t_k - t_{k-1})\big) f\quad \text{for all }f\in \Cb
	\end{align}
	as the composition of the operators $I$ given in \eqref{eq:nonlinear_I} over the partition $\pi$, which represents a multi-period DRO problem over the time interval $[0,t_k]$, see Figure \ref{fig:multi-step.DRO}.\ Again, for the sake of completeness, we define $\mathcal{I}(\{0\})f:=f$.

    Since $I(t)$ maps $\Cb$ to $\Cb$ for all $t\geq 0$, see Proposition \ref{lem:I.on.Cb} (i), the multi-period operator $\mathcal{I}(\pi)$ in \eqref{eq:def_iter} also maps $\Cb$ to $\Cb$ for all $\pi\in \operatorname{P}$.

    \begin{figure}[h]
    \begin{center}
        \def\figelementscolor{black}
    	\begin{tikzpicture}[scale=1]
			font=\fontsize{10}{12}\selectfont
    		\draw [->, thick] (0.5,0) -- (8.5,0);
    		
    		\foreach \x in {1, 2.5,4,6,7.5}
    		{
    			\draw [thick] (\x,0) -- (\x,0.1);
    		}

    		\node [below] at (1,0) {$t_0 = 0$};
    		\node [below] at (2.5,0) {$t_1$};
    		\node [below] at (4,0) {$t_2$};
    		\node [below] at (5,0) {$\cdots$};
    		\node [below] at (6,0) {$t_{k-1}$};
    		\node [below] at (7.5,0) {$t_k$};
    		
    		\draw [thick, \figelementscolor, <-] (1,0.2) to [out=70, in=110]  node [above] {$I(t_1 - t_0)$} (1+1.5,0.2);
    		\draw [thick, \figelementscolor, <-] (2.5,0.2) to [out=70, in=110]  node [above] {$I(t_2 - t_1)$} (2.5+1.5,0.2);
    		\draw [thick, \figelementscolor, <-] (6,0.2) to [out=70, in=110]  node [above] {$I(t_k - t_{k-1})$} (6+1.5,0.2);
    		\node [above, color=\figelementscolor] at (5,0) {$\cdots$};
    	\end{tikzpicture}
        \caption{Multi-period DRO problem as iteration of the one-step DRO operator $I$.}
        \label{fig:multi-step.DRO}
    \end{center}
    \end{figure}
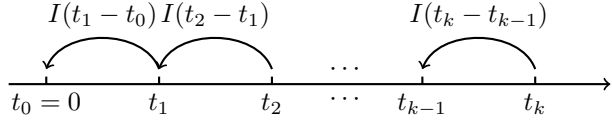
	
	By definition of $\mathcal{I}$ in \eqref{eq:def_iter} and the monotonicity property of $I$ in Proposition \ref{lem:I.on.Cb} (iii), 
    for all $\pi^1,\pi^2\in \operatorname{P}$ with $\pi^2\subset \pi^1$, i.e., $\pi^1$ is finer than $\pi^2$, it holds
	\[
	\mathcal{I}(\pi^1) f \le \mathcal{I}(\pi^2) f\quad \text{for all }f\in \Cb.
	\]
	This allows us to define
	\begin{align} \label{eq:def_S}
		\big(\mathcal{S}(t)f\big)(x) := \inf_{\pi \in \operatorname{P}_t} \big(\mathcal{I}(\pi)f\big)(x)
	\end{align}
    for all $t \ge 0$, $f \in \Cb$, and $x \in \X$, where
    $$\operatorname{P}_t:=\{\pi\in \operatorname{P}\,|\, \max \pi=t\}$$
    is the set of all finite partitions of the interval $[0,t]$ containing $0$ and $t\geq 0$. 

	In a first step, we show that the family $\mathcal{S}:=(\mathcal{S}(t))_{t\geq 0}$, given in \eqref{eq:def_S}, is a strongly continuous monotone semigroup according to the following definition.
	
	\begin{definition} \label{def:semigroup}
			A family $S:=(S(t))_{t \ge 0}$ of continuous operators $S(t) \colon \Cb \to \Cb$ is a {\it strongly continuous monotone semigroup on} $\Cb$ if the following conditions are satisfied:
			\begin{enumerate}
				\item [(i)] $S(0)f=f$ for every $f\in \Cb$.
				\item [(ii)] (Monotonicity) ${S}(t)f \le {S}(t)g$ for all $t \ge 0$ and all $f,g \in \Cb$ with $f \le g$;
				\item [(iii)] (Semigroup property) ${S}(t){S}(s)f = {S}(t+s)f$ for all $t,s \ge 0$ and $f \in \Cb$;
				\item [(iv)] (Strong continuity) the map $t \mapsto {S}(t)f$ is continuous for all $f \in \Cb$. 
			\end{enumerate}
	\end{definition}

        The following theorem is the first main result of this paper.
        
	\begin{theorem}\label{thm:semigroup}
		Suppose that Assumption \ref{ass:general} holds.\ Then, the family $\mathcal{S}$ in~\eqref{eq:def_S} is a strongly continuous monotone semigroup on $\Cb$. Moreover,
        \begin{equation} \label{eq:equivalent.S}
          \cS(t)f = \inf_{\substack{\pi \in \operatorname{P}_t\\\mesh \pi \le h}} \cI(\pi) f  \quad \text{for all } h > 0.
        \end{equation}
	\end{theorem}

    The proof of Theorem \ref{thm:semigroup} can be found in Section \ref{sec:proofs.DRO.multi}. In the following, we briefly describe the key points of the proof.
    Since the properties (i) and (ii) in Definition \ref{def:semigroup} immediately follow from the definition of $\cS$, the proof of Theorem \ref{thm:semigroup} mainly focuses on verifying the semigroup property and the strong continuity of $\cS$. Proving the semigroup property of $\cS$ hinges on the definition of an auxiliary operator $\hat{\cS} \colon \Cb \to \Cb$, see \eqref{eq:S.dyadic}, which represents the multi-period distributionally robust optimization over (almost) dyadic partitions. Leveraging the monotonicity of the operator~$I$ with respect to refining partitions, see Proposition \ref{lem:I.on.Cb}~(iii), and the continuity properties of the best-case maximization problem, which is studied in detail in Section \ref{sec:best-case}, it is possible to show the semigroup property for the operator $\hat{\cS}$ and its strong continuity. In a second step, it is then shown that $\hat{\cS} = \cS$. The restriction to partitions $\pi \in \operatorname{P}_t$ with $\mesh \pi \le h$ in \eqref{eq:equivalent.S}, for arbitrary $h > 0$, then follows from the definition of the operator~$\hat{\cS}$ and its representation as a multi-period optimization over arbitrarily fine dyadic partitions, see \eqref{eq:S.dyadic.inf}.

    We next aim to characterize the infinitesimal generator of the strongly continuous monotone semigroup $\mathcal{S}$, which is defined as follows. 
    \begin{definition}\label{def:semigroup_infinitesimal}
        Let ${S}$ be a strongly continuous monotone semigroup on $\Cb$.\
        Then, the {\it infinitesimal generator} $L\colon D(L)\subset \Cb\to \Cb$ of $S$ is defined by
        \begin{align*}
            D(L)&:=\bigg\{f\in\Cb\,\bigg|\,\,\lim_{t \downarrow 0} \frac{S(t)f- f}{t}\in \Cb\,\mbox{ exists w.r.t.\ the mixed topology}\bigg\},\\
            {L}f&:=\lim_{t \downarrow 0} \frac{{S}(t)f - f}{t}\qquad \mbox{for $f\in D({L})$}.
        \end{align*}
    \end{definition}
    \begin{remark}\label{rem:infinitesimal}
        We consider the following infinitesimal generators in the sense of Definition \ref{def:semigroup_infinitesimal}. 
        \begin{enumerate}
            \item [(i)] Observe that, by Assumption \ref{ass:general} and Theorem \ref{thm:semigroup} (for the special case $m=0$ and $A=\{a\}$),
        the reference semigroup $T^a$ in \eqref{eq:ref_sg} is a strongly continuous monotone semigroup on~$\Cb$ for every~$a\in A$.\ We denote its infinitesimal generator by $\mathscr{L}^a\colon D(\mathscr{L}^a)\subset \Cb \to \Cb$ for all $a\in A$. 
            \item [(ii)] We denote the infinitesimal generator of the strongly continuous monotone semigroup $\mathcal{S}$ in \eqref{eq:def_S}, see Theorem \ref{thm:semigroup}, by $\mathscr{L}\colon D(\mathscr{L})\subset \Cb\to \Cb$.\
        \end{enumerate}
    \end{remark}
	
    The following theorem, which is the main result of this section, provides an explicit description of the infinitesimal generator $\mathscr{L}$ of the strongly continuous monotone semigroup $\mathcal{S}$. 
	\begin{theorem}\label{thm:semigroup.generator}
		Suppose that Assumption \ref{ass:general} holds, and let $(\mathscr{L}^a)_{a\in A}$ and $\mathscr{L}$ be the infinitesimal generators defined in Remark~\ref{rem:infinitesimal}. Then, for every $f \in \cap_{a\in A} D(\mathscr{L}^a)\cap \Cb^1$ with
        \begin{align} \label{eq:ass.gen.bound}
            &\sup_{a \in A} \|\mathscr{L}^a f\|_\infty < \infty\quad \text{and}\\
            &\lim_{t \downarrow 0} \sup_{a \in A} \sup_{x \in K} \bigg| 
            \frac{T^a(t)f(x)-f(x)}{t}-(\mathscr{L}^a f)(x) 
            \bigg| = 0 \quad \mbox{for all $K \Subset \X$},
            \label{eq:ass.generator}
        \end{align}
		it holds that $f\in D(\mathscr{L})$ and 
		\begin{align*}
		\mathscr{L}f = \inf_{a \in A}\mathscr{L}^a f + m \|\nabla f(\,\cdot\,)\|.
		\end{align*}
	\end{theorem}
	The proof of Theorem \ref{thm:semigroup.generator} is presented in Section \ref{sec:proofs.DRO.multi}. The main challenge is to show that the generator of the semigroup actually coincides with the derivative $\lim_{t \downarrow 0} \frac{I(t)f - f}{t}$ of the one-period operators at zero for functions $f$ as in Theorem \ref{thm:semigroup.generator}. The result is then a direct consequence of Proposition \ref{lem:generator}, splitting the generator of the DRO into the sum of the generator of the non-robust dynamics and a component arising from uncertainty according to
    \[
    \frac{I(t)f - f}{t} = \frac{T(t)f - f}{t} + \frac{I(t)f - T(t)f}{t}.
    \]    
    We briefly discuss the uniformity conditions \eqref{eq:ass.gen.bound} and \eqref{eq:ass.generator}.\ Both conditions resemble boundedness in the way the action enters the controlled dynamics.\ In practice, Condition \eqref{eq:ass.generator} can be verified using, for example, It\^o's lemma.\ Alternatively, using \cite[Proposition 2.4 (2)]{kunze}, Condition \eqref{eq:ass.generator} is equivalent to
    \begin{equation}\label{eq.explanation}
     \lim_{t\downarrow 0}\sup_{a\in A}\sup_{x\in K}\bigg|\E\bigg[\frac1t\int_0^t (\mathscr L^af)(X_s^{x,a})-(\mathscr L^af)(x)\, \d s\bigg]\bigg|=0 \quad \mbox{for all $K \Subset \X$},
    \end{equation}
    where $\mathscr L^a$ is the infinitesimal generator of the Markov process $X^a$ associated with the action, i.e., constant control $a\in A$.\ If, additionally, $\mathscr L^a f\in D(\mathscr L^a)$ for all $a\in A$ and
    \begin{equation}\label{eq.explanation2}
    \sup_{a\in A}\big\| \mathscr L^a\mathscr L^af\big\|_\infty<\infty,
    \end{equation}
    then \eqref{eq.explanation} yields
    \[
    \sup_{a\in A}\sup_{x\in K}\bigg|\E\bigg[\frac1t\int_0^t (\mathscr L^af)(X_s^{x,a})-(\mathscr L^af)(x)\, \d s\bigg]\bigg|\leq \frac{t}{2} \sup_{a\in A}\big\| \mathscr L^a\mathscr L^af\big\|_\infty\quad \text{for all }t>0.
    \]
    Hence, \eqref{eq:ass.generator} is satisfied in this case.\ Observe that $\mathscr L^a f\in D(\mathscr L^a)$ for all $a\in A$ corresponds to higher regularity of $f$ and \eqref{eq.explanation2} again resembles a uniformity over the controls.
    \section{Connection to nonlinear expectations and robust optimization under drift uncertainty.}\label{sec:apply}
    
    In this section, we connect the strongly continuous monotone semigroup $\mathcal{S}$, obtained as the scaling limit of the multi-period DRO problems, see Theorem\;\ref{thm:semigroup}, with nonlinear PDEs of Hamilton-Jacobi-Bellman-Isaacs-type that arise from nonlinear expectations and robust optimization problems under drift uncertainty.

    \subsection{Nonlinear PDEs arising from the semigroup $\cS$ and their viscosity solutions}
    In order to establish the connection between the semigroup $\cS$ and a nonlinear PDE, the infinitesimal generator $\mathscr{L}$ of $\cS$, see Remark\;\ref{rem:infinitesimal} (ii), plays a crucial role. For given $u_0\in \Cb$, we consider the nonlinear PDE
    \begin{align}\label{eq:viscosity}
        \left\{
        \begin{aligned}
        &\partial_tv(t,x) = \mathscr{L} v(t,x)\quad &&\mbox{for all }(t,x)\in (0,\infty)\times \X,\\
        &v(0,x)= u_0(x)\quad && \mbox{for all }x\in \X.
        \end{aligned}
        \right.
    \end{align}

    In what follows, we introduce the notion of a viscosity solution to \eqref{eq:viscosity}, as considered in \cite{fuchs2025existence}.
    \begin{definition}\label{def:D.M.viscosity}
        Let $u_0\in \Cb$ and $D \subseteq D(\mathscr{L})$ be nonempty. 
        We say that
        $v\colon [0,\infty)\times \X \to \mathbb{R}$ is a $D$-viscosity subsolution or supersolution to \eqref{eq:viscosity} if $v\in \Cb([0,\infty)\times \X)$, $v(0,x)= u_0(x)$ for all $x\in \X$ and, for all $(t,x)\in (0,\infty)\times \X$ and $\varphi\in \Cb^1((0,\infty)\times \X)$ with $\varphi(t,\cdot)\in D$, $v(t, x)=\varphi(t, x)$, and $v \leq \varphi$ or $v\geq \varphi$, it holds
        \[
        \quad \partial_t\varphi(t,x)\leq \mathscr{L}\varphi(t, x)\quad \mbox{or}\quad\partial_t\varphi(t, x)\geq \mathscr{L}\varphi(t, x),\quad\text{respectively}.\footnote{We write $\Cb([0,\infty)\times \X)$ for the set of all continuous bounded functions on $[0,\infty)\times \X$. Moreover, we write $\Cb^1((0,\infty)\times \X)$ for the set of all bounded functions on $(0,\infty)\times \X$ having bounded continuous derivative.}
        \]
        We say $v$ is a $D$-viscosity solution to \eqref{eq:viscosity} if it is both a $D$-viscosity sub- and supersolution.
    \end{definition}

    If the set $D$ contains $\Cb^2$ as a subset, then the notion of $D$-viscosity solutions in Definition \ref{def:D.M.viscosity} is stronger than classical notions of viscosity solutions presented, for example, in \cite{crandall1992user,crandall1983viscosity,fleming2006controlled,neufeld2017nonlinear} and slightly weaker than the notion of a $D$-$M$ viscosity solution, considered in \cite{fuchs2025existence}.
    For a detailed discussion, we refer to \cite[Remark 2.11]{fuchs2025existence}.

    In order to establish the existence of a $D$-viscosity solution to the PDE in \eqref{eq:viscosity}, we consider the following property for operators on $\Cb$, introduced in \cite{fuchs2025existence}.
        
        \begin{definition}\label{def:K_convex}
        Let $K:=(K(t))_{t\geq 0}$ be a family of operators $\Cb\to\Cb$. A family $S:=(S(t))_{t\geq 0}$ of operators $\Cb\to\Cb$ is called {\it $K$-convex} if for all $\lambda \in [0,1]$, $f,g\in \Cb$, and  $t\geq 0$,
        \[
            S(t)\big(\lambda f+ (1-\lambda)g\big)\leq \lambda S(t)f+ (1-\lambda)K(t)g.
        \]
    \end{definition}    
    
    \begin{theorem}\label{cor:viscosity} Suppose that Assumption \ref{ass:general} holds.\
    Let $\mathcal{S}$ be the strongly continuous monotone semigroup on $\Cb$ defined in~\eqref{eq:def_S}, $\mathscr{L}$ be its infinitesimal generator, see Remark~\ref{rem:infinitesimal} (ii), and $D\subseteq D(\mathscr{L})$ be nonempty. Then, for every $u_0\in \Cb$, the function $v \colon [0,\infty)\times \X\to \mathbb{R}$, defined~by
    \begin{align}\label{eq:viscosity_semi}
    v(t,x):=\big(\mathcal{S}(t)u_0\big)(x)\quad\text{for all }t\geq 0\text{ and }x\in \X,
    \end{align}
     is a $D$-viscosity solution to the PDE \eqref{eq:viscosity}.
    \end{theorem}

    The proof of Theorem \ref{cor:viscosity}, which is contained in Section \ref{sec:proofs.sec.appl}, is based on the construction of a suitable family $K:=(K(t))_{t\geq 0}$ of operators $\Cb\to\Cb$ in order to apply \cite[Theorem 3.1]{fuchs2025existence} to establish the existence of a $D$-viscosity solution to the PDE \eqref{eq:viscosity}.

    \begin{remark} Let $D\subset\bigcap_{a\in A}D(\mathscr{L}^a)\cap\Cb^1$ be a non-empty subset of functions satisfying \eqref{eq:ass.gen.bound} and \eqref{eq:ass.generator}. Then, by Theorem \ref{thm:semigroup.generator} and Theorem~\ref{cor:viscosity}, the function $v\colon [0,\infty)\times \X\to  \R$ defined in \eqref{eq:viscosity_semi} is a $D$-viscosity solution~to 
            \begin{align*}
        \hspace{3.em} \left\{
        \begin{aligned}
            &\partial_tv(t,x)=\inf_{a \in A}\mathscr{L}^a v(t,x) + m \|\nabla v(t,x)\|\quad &&\text{for all }(t,x)\in(0,\infty)\times \X, \\
            &v(0,x)=u_0(x) &&\text{for all }x\in \X.
        \end{aligned}
        \right.
        \end{align*}
    \end{remark}

    \subsection{Connection to $g$-expectation}
    The result from the preceding section allows us to establish a connection between the semigroup $\mathcal{S}$ and the $g$-expectation framework, see \cite{peng1997backward,coquet2002filtration}, which we discuss in the sequel.

    Let $(\Omega,\mathcal{F}, \mathbb{P})$ be a complete probability space, $W:=(W_t)_{t\geq 0}$ be a $d$-dimensional standard Brownian motion with its Brownian standard filtration $\mathbb{F}:=(\mathcal{F}_t)_{t\geq0}$, i.e., the natural filtration generated by $W$ augmented by all $\mathbb P$-null sets, and $T>0$ be a finite time horizon.\
    For any probability measure $\mathbb{Q}$ on $(\Omega,\mathcal{F}_{T})$, we write $\mathbb{E}^{\mathbb{Q}}[\,\cdot\,]$ for the expectation under $\mathbb{Q}$. 
    \begin{definition}\label{dfn:g_exp}
        We define the $g$-expectation $\mathcal{E}\colon  L^2(\mathcal{F}_{T})\ni \xi\mapsto \mathcal{E}[\xi]\in \mathbb{R}$~by\footnote{We write $L^2(\mathcal{F}_{T})$ for the set of all real-valued, $\mathcal{F}_{T}$-measurable random variables $\xi$ with $\mathbb{E}^{\mathbb{P}}[|\xi|^2]<\infty$,  $\mathbb{L}^2(\mathbb{R}^d)$ for the set of all $\mathbb{R}^d$-valued, $\mathbb{F}$-predictable processes $Z=(Z_t)_{t\in[0,T]}$ with $\mathbb{E}^{\mathbb{P}}[\int_0^{T_1}\|Z_t\|^2\,\d t]<\infty$, and $\mathbb{S}^2(\mathbb{R})$ for the set of all real-valued, $\mathbb{F}$-progressively measurable c\`adl\`ag (i.e., right-continuous with left limits) processes $Y=(Y_t)_{t\in[0,T]}$ with $\mathbb{E}^{\mathbb{P}}[\sup_{t\in[0,T]}|Y_t|^2]<\infty$.}
	\begin{align*}
		\mathcal{E}[\xi]:=Y_0,
	\end{align*}
	where $(Y,Z)\in \mathbb{S}^2(\mathbb{R})\times \mathbb{L}^2(\mathbb{R}^d)$ is the unique solution to the backward stochastic differential equation (BSDE), see \cite[Theorem~3.1]{pardoux1990adapted},
	\begin{align*}
		Y_t= \xi + \int_t^{T} m\|Z_s\| \,\d s -\int_t^{T} Z_s^\top \, \d W_s.
	\end{align*}
    Moreover, its conditional $g$-expectation with respect to $\mathcal{F}_t$ is defined by
	\[
	\mathcal{E}[\xi \,|\, \mathcal{F}_t]:=Y_t\quad \mbox{for all }t\in [0,T].
	\]
    \end{definition}

    \begin{remark}\label{rem:g_exp}
	By \cite[Proposition 3.6]{el1997backward}, the $g$-expectation in Definition \ref{dfn:g_exp} satisfies the following representation for every  $\xi \in L^2(\mathcal{F}_T)$ and $t\in[0,{T}]$:
	\begin{align}\label{eq:g_exp_vari}
	   \mathcal{E}[\xi \,| \, \mathcal{F}_t] = \esssup_{\vartheta \in \mathcal{C}(m)} \mathbb{E}^{\mathbb{P}^\vartheta}[\xi \,| \, \mathcal{F}_t],
	\end{align}
	where $\mathcal{C}(m)$ is the set of all $\mathbb{R}^d$-valued, $\mathbb{F}$-progressively measurable processes $\vartheta=(\vartheta_t)_{t\in[0,T]}$ with $|\vartheta_t|\leq m$ $\mathbb{P}\otimes \d t$-a.e., and  for each $\vartheta\in \mathcal{C}(m)$ the probability measure $\mathbb{P}^\vartheta$ on $(\Omega,\mathcal{F}_{T})$ is defined by
	\[
		\frac{\d \mathbb{P}^\vartheta}{\d \mathbb{P}|_{\mathcal{F}_{T}}}:= \exp\bigg(-\frac{1}{2}\int_0^{T}\|\vartheta_s\|^2\,\d s+\int_0^{T}\vartheta_s^\top \,\d W_s\bigg).
	\]
    By the variational representation in \eqref{eq:g_exp_vari}, the conditional $g$-expectation $(\mathcal{E}[\,\cdot\,|\mathcal F_t])_{t\in [0,T]}$
    is often considered to be a dynamic risk measure, see, e.g., \cite{hernandez2007control,schied2006risk, maccheroni2006ambiguity,schied2007optimal,talay2002worst,park2022robust,park2023robust}.
    \end{remark}

    We now consider the following $g$-expectation value function defined w.r.t.\ a given $h\in \Cb$. Set
    \begin{align}\label{eq:value_g_exp}
        V(t,x):= \mathcal{E}\big[h\big( X_{T}^{t,x}\big)\,\big|\,\mathcal{F}_t\big],\quad \mbox{for $(t, x)\in[0,T_1]\times \mathbb{R}^d$,}
    \end{align}
    where
     $X_T^{t,x}:=
	 x+ b(T-t) +(W_T-W_t)$
    with  $ b\in \X$. 

    The following corollary establishes the connection between the strongly continuous monotone semigroup $\cS$ with the $g$-expectation value in \eqref{eq:value_g_exp}. The proof can be found in Section\;\ref{proof:apdx:cor:g_exp}.
    
    \begin{corollary}\label{cor:g_exp}
        Let $(\psi^a)_{a\in A}$ and $(Y^a)_{a\in A}$ be given by \eqref{eq:eg_psi} and \eqref{eq:ito_diffusion} in Example \ref{lem:examples} a), respectively, with $b(a):= b\in \X$ and $\sigma(a):=I_d$, where $I_d$ is the identity matrix in $\mathbb{R}^{d\times d}$.\ Then, the following hold:
        \begin{itemize}
            \item [(i)] The inclusion $\Cb^2\subseteq D(\mathscr{L}^a)$ holds for all $a\in A$, and every $f\in\Cb^2$ satisfies \eqref{eq:ass.gen.bound} and \eqref{eq:ass.generator}.
            \item [(ii)] The value function $V$ in \eqref{eq:value_g_exp} coincides with the function $v\colon [0,T]\times \X\to \mathbb{R}$, defined by
            \begin{align}\label{eq:viscosity_g_exp_value}
            v(t,x):=(\cS(T-t)h)(x)\quad \text{for all }(t,x)\in [0,T]\times \X,
            \end{align}
            and $V=v$ is the unique viscosity solution\footnote{\label{footnote:viscosity} In the classical sense of \cite{crandall1992user,crandall1983viscosity,fleming2006controlled}.} 
            to the following PDE:
            \begin{align}\label{eq:viscosity_g_exp}
            \left\{
                \begin{aligned}
                    &\partial_t v(t,x)+\frac{1}{2}\Delta v(t,x) + \langle b, \nabla v(t,x) \rangle+ m\|\nabla v(t,x)\|=0\quad &&\text{for all }(t,x)\in[0,T)\times \X,\\
                    &v(T,x)=h(x) &&\text{for all } x\in \X.
                \end{aligned}
            \right.
            \end{align}
        \end{itemize}
    \end{corollary}

    \begin{remark}
        In view of Remark~\ref{rem:g_exp} and Corollary~\ref{cor:g_exp}, the function $v$ in \eqref{eq:viscosity_g_exp_value} can also be seen as the unique viscosity solution of a nonlinear Kolmogorov equation under drift uncertainty, see, e.g., \cite{neufeld2017nonlinear, bartl2024numerical, neufeld2014measurability, liu2019compactness}, or, alternatively, as the solution to a Hamilton–Jacobi–Bellman (HJB) equation arising from stochastic control problems with drift control, see, e.g., \cite{fleming2006controlled, pham2009continuous, yong1999stochastic}.\ For similar observations as well as related settings involving drift and volatility uncertainty (or control), cf.\ \cite{bartl2023sensitivity,denis2013optimal,neufeld2018robust,tevzadze2013robust,pham2022portfolio,bartl2021duality,park2024irreversible}, corresponding to second-order backward stochastic differential equations, cf.\ \cite{cheridito2007second,matoussi2015robust,soner2012wellposedness,soner2013dual}, within the framework of nonlinear semigroups, we refer to \cite{BK22, nendel2021wasserstein, BKN23, denk2020semigroup, nendel2021upper, blessing2024discrete}.
    \end{remark}

    \subsection{Connection to robust optimization under drift uncertainty}
    In this section, we aim to establish the connection between the semigroup $\mathcal{S}$ and robust optimization problems under drift uncertainty, which will be formulated within the framework of stochastic differential games, see, e.g., \cite{friedman1972stochastic,fleming1989existence,hamadene1997bsdes,buckdahn2008stochastic,talay2002worst}.\ To that end, we consider a complete probability space $(\Omega,\mathcal F,\P)$ supporting a $d$-dimensional standard Brownian motion $(W_t)_{t\geq 0}$ together with its natural filtration $\mathbb F:=(\mathcal{F}_t)_{t\geq0}$ augmented by all $\P$-null sets.

    Assume there exists a decision maker who aims to minimize a terminal cost by choosing an optimal control (so that they act as the minimizer), while the environment systematically acts in opposition to the decision maker’s interest (so that it plays the role of the maximizer in an adversarial way).
    
    Fix $m\geq0$, a nonempty set of actions $A \Subset \X$, and a time-horizon $T>0$. We then consider the following set of admissible controls for the environment
    \begin{align}\label{eq:control_env}
        \Lambda_t:=\big\{\vartheta:=(\vartheta_s)_{s\in [t,T]}\,\big|\,
            \vartheta\text{ is $\mathbb{F}$-progressively measurable and }\|\vartheta_s\|\leq m \;\P\otimes \d s\text{-a.e.}\big\}
    \end{align}
 and the set of admissible controls for the decision maker
    \begin{align*}
        \mathcal{A}_t:=\big\{\alpha=(\alpha_s)_{s\in [t,T]} \,\big|\;\alpha \text{ is $\mathbb{F}$-progressively measurable and }\alpha_s\in A\;\P\otimes \d s\text{-a.e.}\big \}.
    \end{align*}

    We now consider the robust optimization problem of the decision maker w.r.t.\ a given function $h\in \Cb$. Set
    \begin{align}\label{eq:value_SDG}
        V(t,x):=\inf_{\alpha\in \mathcal{A}_t}\sup_{\vartheta \in \Lambda_t} \mathbb{E}^{\mathbb{P}}\Big[h\big( X_{T}^{t,x;\alpha,\vartheta}\big)\Big]\quad \text{for all }(t,x)\in [0,T]\times \X,
    \end{align}
    where, for all $(t,x)\in [0,T]\times \X$ and all 
    $(\alpha,\vartheta)\in \mathcal{A}_t\times \Lambda_t$, we consider a controlled Brownian motion $X^{t,x;\alpha,\vartheta}:=(X_s^{t,x;\alpha,\vartheta})_{s\in [t,T]}$ starting at time $t$ in $x$, given by
    \begin{align}\label{eq:control_process}
    X_s^{t,x;\alpha,\vartheta}:= x+\int_t^s\big(b(\alpha_u)+\vartheta_u\big)\,\d u + \int_t^s\sigma(\alpha_u)\,\d W_u\quad\text{for } s\in[t,T],
    \end{align}
    with bounded functions $b\colon A\to\X$ and $\sigma\colon A\to \mathbb{R}^{d\times d}$.
    
    The following corollary establishes the connection between the strongly continuous monotone semigroup $\cS$ and the value function $v$ of the robust optimal control problem in \eqref{eq:value_SDG}. The proof is contained in Section\;\ref{proof:apdx:cor:SDG}.
    \begin{corollary}\label{cor:SDG}
        Let $(\psi^a)_{a\in A}$ and $(Y^a)_{a\in A}$ be given by \eqref{eq:eg_psi} and \eqref{eq:ito_diffusion} in Example \ref{lem:examples} a), respectively, with bounded functions $b\colon A\to\X$, and $\sigma\colon A\to \mathbb{R}^{d\times d}$.\ Then, the following hold:
        \begin{itemize}
            \item [(i)] $\Cb^2\subseteq \cap_{a\in A} D(\mathscr{L}^a)$ and every $f\in\Cb^2$ satisfies \eqref{eq:ass.gen.bound}\;and\;\eqref{eq:ass.generator} in Theorem\;\ref{thm:semigroup.generator}.
            \item [(ii)] The value function $V$ given in \eqref{eq:value_SDG} coincides with the function $v\colon [0,T]\times \X\to \mathbb{R}$, defined by
            \begin{align}\label{eq:viscosity_SDG_value}
            v(t,x):=(\cS(T-t)h)(x)\quad \text{for all }(t,x)\in [0,T]\times \X,
            \end{align}
            and $V=v$ is the unique viscosity solution to the following PDE:
            \begin{align}\label{eq:HJBI}
    \left\{
        \begin{aligned}
    &\partial_tv(t,x)+\mathscr{L}v(t,x)=0\quad &&\text{for all }(t,x)\in [0,T)\times \X,\\
    &v(T,x)= h(x)\qquad && \text{for all }x\in \X,
        \end{aligned}
    \right.
    \end{align}
    where the infinitesimal generator $\mathscr{L}$ is given by 
    \[
    \mathscr{L}v(t,x)
        = \inf_{a\in A}\bigg\{\frac{1}{2}\operatorname{tr}\big(\sigma(a)\sigma(a)^\top \nabla^2 v(t,x) \big) + \big\langle b(a), \nabla v(t,x) \big\rangle \bigg\}+m\|\nabla v(t,x)\|.
    \] 
        \end{itemize}
    \end{corollary}

    \begin{remark}
        While the robust optimization problem in \eqref{eq:value_SDG} is given as a min--max problem, the resulting PDE in \eqref{eq:HJBI} takes the form of a HJB equation involving the minimization over controls, i.e., $\inf_{a \in A}$, rather than a full Isaacs-type equation.\ This reduction stems from the specific structure of the uncertainty in \eqref{eq:control_env} and the controlled process in \eqref{eq:control_process}: the uncertainty appears only in the drift term and is additively separated from the control process.\ As a consequence, the associated Hamilton-Jacobi-Bellman-Isaacs (HJBI) equation for the problem \eqref{eq:value_SDG}, see \eqref{eq:HJBI2} in Section \ref{proof:apdx:cor:SDG}, reduces to an HJB equation with an additive nonlinear first-order term in this setup.
    \end{remark}

\section{Best-case maximization}
\label{sec:best-case}
In this section, we discuss the `best-case' maximization problem, where the infimum over the action set $A$ is replaced by a supremum in the definition of the family of operators $I$.\ The properties of the resulting operators are essential to prove our main results Theorem \ref{thm:semigroup} and Theorem \ref{thm:semigroup.generator}.
We consider the family $J=(J(t))_{t\geq0}$ of operators, defined by 
\begin{equation} \label{eq:def.J}
 \big(J(t)f\big)(x):=\sup_{a\in A}\sup_{\nu\in \cB_t^a(m)}\int_{\R^d}f\big(\psi_t^a(x)+z\big)\,\nu(\d z)
\end{equation}
for all $t\geq 0$, $f\in\Cb$, and $x\in \R^d$, so that $J(t)$ corresponds to the `best-case' maximization problem for the ambiguity set $({\mathcal{B}}_t^a(m))_{a\in A}$ for $t\geq0$.\ Then, similarly to what has been done for the family $I$, for $\pi = \{ t_0, \dots, t_k \} \in \operatorname{P}$ with $0=t_0<\ldots<t_k$ and $k\in \N$, we consider the composition 
	\begin{align} \label{eq:def_iter.J}
		\mathcal{J}(\pi)f := \big(J(t_1 - t_0) \cdots J(t_k - t_{k-1})\big) f\quad \text{for all }f\in \Cb.
	\end{align}

\begin{remark} \label{rem:J}
    Note that,  for all $t \ge 0$, $f \in \Cb$, $\|J(t)f\|_\infty\leq \|f\|_\infty$ and, for all $t \ge 0$, $f,g \in \Cb$, and $x \in \X$, it holds
    \[
    |(J(t)f)(x) - (J(t)g)(x)| \le \sup_{a \in A} \sup_{\nu \in {\mathcal{B}}_t^a(m)} \int_\X \big|f\big(\psi_t^a(x) + z\big) - g\big(\psi_t^a(x) + z\big)\big|\,\nu(\d z) \le \|f - g\|_\infty.
    \]
    Therefore, the operator $J(t)$ is a contraction for each $t \ge 0$, i.e.,
    \begin{equation} \label{eq:J.contraction}
        \|J(t)f - J(t)g\|_\infty \le \|f - g\|_\infty \quad \text{for all } t \ge 0 \text{ and } f,g \in \Cb.
    \end{equation}    
    A direct computation shows that $J(t)$ is subadditive and positive homogeneous for each $t \ge 0$, i.e.,
    \[
    \big(J(t)(f + g)\big)(x) \le (J(t)f)(x) + (J(t)g)(x) \quad \text{and} \quad \big(J(t)(\lambda f)\big)(x) = \lambda (J(t)f)(x)    \]
    for all $f,g \in \Cb$, $\lambda \ge 0$, and $x \in \X$.
\end{remark}

We then have the following continuity results, which play a crucial role in the subsequent discussion.

\begin{lemma} \label{lem:J.cont.above}
	Suppose that Assumption \ref{ass:general} holds.\ Then, for all $\eps>0$, $t\ge 0$, and $K \Subset \X$, there exist $h>0$ and $K' \Subset \X$ such that, for all $f, g \in \Cb$ and all partitions $\pi \in \operatorname{P}$ with $\max \pi \le t$ and $\mesh \pi \le h$,
	\[
	\|{\mathcal{J}}(\pi)f - {\mathcal{J}}(\pi)g\|_{\infty, K} \le \eps \|f-g\|_\infty+ \|f - g\|_{\infty, K'}.
	\]
\end{lemma}

\begin{proof}
	Note that, for all $\pi \in \operatorname{P}$, $\cJ(\pi)$ is monotone, subadditive, and positive homogeneous, since it is the composition of the monotone, subadditive, and positive homogeneous operators $J$, see Remark \ref{rem:J}.\ Moreover, $\cJ(\pi)0 = 0$ and $\cJ(\pi)1=1$ for all $\pi \in \operatorname{P}$.\ We will show the statement by employing Lemma~\ref{lem:crit_above} and Lemma~\ref{lem:cut-off-existence}.\ Using \eqref{eq.le.gen.1}, for all $s\geq 0$ and $\zeta \in \Cci$, we obtain
	\begin{align*}
		\|J(s)(-\zeta) - (-\zeta)\|_\infty & \le \sup_{a \in A} \|I^a(s)(-\zeta) - T^a(s)(-\zeta)\|_\infty + \sup_{a \in A} \| T^a(s)(-\zeta) - (-\zeta) \|_\infty \\
		& \leq s\sup_{x\in \X}\bigg(\sup_{a\in A}\frac{\big(I^a(s)(-\zeta)\big)(x)-\big(T^a(s)(-\zeta)\big)(x)}{s}\bigg) \\
		& \qquad + \sup_{a \in A} \|T^a(s) (-\zeta) -(-\zeta)\|_\infty\\
		&\leq sm \|\nabla \zeta\|_\infty + \sup_{a \in A} \|T^a(s) (-\zeta) -(-\zeta)\|_\infty.
	\end{align*}
    Now, let $\eps>0$, $t\ge 0$, and $K \Subset \X$.\ Since $K$ is compact, there exists a radius $R > 0$ such that $K \subset \{x\in \X \,|\, \|x\|\leq R\}$.\	Moreover, let $\delta>0$ such that
    $(m+C'_\eps)\delta <\frac{\eps}{2t}$, where $C_\eps'\geq0$ is the constant from Lemma \ref{lem:apriori.linsg}.\ Then, by Lemma \ref{lem:cut-off-existence}, there exists a function $\zeta \in \Cci$ with $0 \le \zeta \le 1$ and $\zeta(x) = 1$ for all $x \in K$, which satisfies
	\begin{equation} \label{eq:pr.select.zeta}
	\sup_{y \in \X} (1 + \|y\|)\|\nabla \zeta(y)\| + \|\nabla^2 \zeta\|_\infty \le \delta.
	\end{equation}
    Moreover, by Lemma \ref{lem:apriori.linsg}, there exists $h>0$ such that
	\[
    \sup_{a \in A} \|T^a(s) (-\zeta) -(-\zeta)\|_\infty \le s\bigg[\frac{\eps}{2t} + C'_\eps \bigg( \sup_{y \in \X} (1+ \|y\|)\|\nabla \zeta(y)\| + \|\nabla^2 \zeta\|_\infty \bigg) \bigg]\quad \text{for all }s\in (0,h].
	\]
	We thus obtain that
	\[
	\|J(s)(-\zeta) - (-\zeta)\|_\infty \le s \big( \tfrac{\eps}{2t}  + (m + C_\eps')\delta \big)<s\frac{\eps}t \quad \text{for all } s \in (0,h].
	\]
    Next, consider a partition $\pi = \{t_0, t_1, \dots, t_k \}\in \operatorname{P}$ with $\max \pi\leq t$ and $\mesh \pi\leq h$.\ Using \eqref{eq:J.contraction} and the translation invariance of $J(s)$ for all $s \ge 0$, we find that, for all $x \in K$,
	\begin{align}
		\big(\cJ(\pi) (1 - \zeta)\big)(x)&=\big(\cJ(\pi) (1 - \zeta)\big)(x)-(1-\zeta)(x)\leq \| \cJ(\pi)(1 - \zeta) - (1 - \zeta)\|_\infty \nonumber\\
        &= \|\cJ(\pi)(-\zeta) - (-\zeta)\|_\infty 
	\leq \sum_{i=1}^k \|J(t_i-t_{i-1})(-\zeta) - (-\zeta)\|_\infty \nonumber\\
    &\leq \frac{\eps}t\sum_{i=1}^k (t_i-t_{i-1})\leq \eps, \label{eq:telescopic.J}
	\end{align}
	where, in the first step, we used the fact that $\zeta(x) = 1$ for all $x\in K$, so that $1-\zeta(x) = 0$ for $x\in K$.\
    Therefore, the function $\zeta$ satisfies the conditions (i)-(iii) in Lemma \ref{lem:crit_above} with $K':=\operatorname{supp}\zeta$, and the result follows.
\end{proof}

 \begin{lemma} \label{lem:right.continuity}
     Suppose that Assumption \ref{ass:general} holds.\ Then,
     \[
     \lim_{n\to \infty}\sup_{x\in K}\sup_{\pi\in \operatorname{P}_{t_n}}\big|\big(\mathcal J(\pi)f_n\big)(x)- f(x)\big|=0\quad \text{for all }K\Subset \Rd
     \]
     for all sequences $(t_n)_{n\in \N}\subset [0,\infty)$ with $t_n\to 0$ and $(f_n)_{n\in \N}\subset \Cb$ with $f_n\tomix f\in \Cb$.
 \end{lemma}

 \begin{proof}
    Let $(t_n)_{n\in \N}\subset [0,\infty)$ with $t_n\to 0$ and $(f_n)_{n\in \N}\subset \Cb$ with $f_n\tomix f\in \Cb$.\ Moreover, fix $\eps > 0$ and $K \Subset \X$.\ Then, by Lemma \ref{lem:J.cont.above} and since $t_n\to 0$, there exists $K'\Subset \Rd$ and $n_0\in \N$ such that
    \[
    \sup_{x\in K}\sup_{\pi\in \operatorname{P}_{t_n}}\big|\big(\mathcal J(\pi)f_n\big)(x)- f(x)\big|\leq \frac\eps2 \|f-f_n\|_\infty+\|f-f_n\|_{\infty,K'}+ \sup_{\pi\in \operatorname{P}_{t_n}}\|\mathcal J(\pi)f-f\|_{\infty,K}
    \]
   for all $n\in \N$ with $n\geq n_0$.\ Since $f_n\tomix f$, the claim follows once we have proved that $$\lim_{n\to \infty}\sup_{\pi\in \operatorname{P}_{t_n}}\|\mathcal J(\pi)f-f\|_{\infty,K}=0.$$
   To that end, let $\zeta\in \Cci$ with $\|\zeta\|_\infty\leq \|f\|_\infty+1$ and $\|f-\zeta\|_{\infty,K'}\leq \frac\eps2$, cf. Lemma \ref{lem:approx.res}.\ Then, by Lemma~\ref{lem:J.cont.above},
    \[
     \sup_{\pi\in \operatorname{P}_{t_n}}\|\mathcal J(\pi)f-f\|_{\infty,K}\leq \eps\big(1+\|f\|_\infty\big)+\sup_{\pi\in \operatorname{P}_{t_n}}\|\mathcal J(\pi)\zeta-\zeta\|_{\infty,K}
    \]
    for all $n\in \N$ with $n\geq n_0$.\ We conclude the proof by showing that there exist $\alpha > 0$ and $h > 0$ such that
    \begin{equation} \label{eq:C2c.lip}
        \sup_{\pi \in \operatorname{P}_t}\|\cJ(\pi)\zeta - \zeta\|_\infty \le \alpha t \quad \text{for all } t \in [0, h].
    \end{equation}
    As a matter of fact, thanks to Lemma \ref{lem:apriori.linsg}, there exists $\beta >0$ and $h > 0$ such that
    \begin{equation} \label{eq:C2c.lip.1}
        \sup_{a \in A} \sup_{x \in \X} |(T^a(t)\zeta)(x) - \zeta(x)| \le \beta t \quad \text{for all } t \in [0, h].
    \end{equation}
    Moreover, using \eqref{eq.le.gen.1}, we also obtain
    \begin{equation} \label{eq:C2c.lip.2}
        \sup_{a \in A} \sup_{x \in \X} |(I^a(t)\zeta)(x) - (T^a(t)\zeta)(x)| \le m \|\nabla \zeta\|_\infty t \quad \text{for all } t \geq 0.
    \end{equation}
    Therefore, putting together \eqref{eq:C2c.lip.1} and \eqref{eq:C2c.lip.2}, there exists $\alpha > 0$ such that
    \begin{equation} \label{eq:C2c.lip.3}
        |(J(t)\zeta)(x) - \zeta(x)| \le \sup_{a \in A}|(I^a(t)\zeta)(x) - (T^a(t)\zeta)(x)| + \sup_{a \in A} |(T^a(t)\zeta)(x) - \zeta(x)| \le \alpha t
    \end{equation}
    for all $t \in [0, h]$ and all $x \in \X$.\ Now, \eqref{eq:C2c.lip} is obtained by iterating \eqref{eq:C2c.lip.3} and using \eqref{eq:J.contraction} as in~\eqref{eq:telescopic.J}.
\end{proof}

For the proof of Theorem \ref{cor:viscosity}, we also consider the operator 
\begin{equation}\label{eq.def.K}
 K(t)f:= \inf_{h>0}\sup_{\substack{\pi\in \operatorname{P}_t\\ \mesh\pi\leq h}}\cJ(\pi)f=\lim_{h\downarrow0}\sup_{\substack{\pi\in \operatorname{P}_t\\ \mesh\pi\leq h}}\cJ(\pi)f\quad \text{for all }f\in \Cb\text{ and }t\geq 0.
\end{equation}
We have the following auxiliary result.
\begin{lemma}\label{lem:Kconvex}
 Suppose that Assumption \ref{ass:general} holds.\ Then, for all $t\geq 0$, $K(t)\colon \Cb\to \Cb$ is well defined, continuous, and satisfies
 \[
K(t_n)f_n\tomix f
\]
for all sequences $(t_n)_{n\in \N}\subset [0,\infty)$ with $t_n\to 0$ and $(f_n)_{n\in \N}\subset \Cb$ with $f_n\tomix f\in \Cb$.
\end{lemma}

\begin{proof}
Let $f\in \Lipb$ and $t\geq 0$. Then, for all $x_1,x_2\in \X$,
\begin{align}
\notag  \sup_{a \in A} \sup_{\nu \in {\mathcal{B}}_t^a(m)} \int_\X \big|f\big(\psi_t^a(x_1) + z\big) - f\big(\psi_t^a(x_2) + z\big)\big|\,\nu(\d z)
        &\le \sup_{a \in A} \|f\|_\Lip \big\|\psi_t^a(x_1) - \psi_t^a(x_2)\big\| \\ &\le e^{ct} \|f\|_\Lip \|x_1 - x_2\|, \label{eq.lipschitz.apriori}
    \end{align}
    where the last inequality follows from Assumption \ref{ass:general}\;(ii).\ This implies that $J(t)f\in \Lipb$ with $\|J(t)f\|_\Lip\leq e^{ct}\|f\|_\Lip$.\ Iterating this estimate and using the fact that $\|K(t)g\|_\infty\leq \|g\|_\infty$ for all $g\in \Cb$, yield that $K(t)\colon \Lipb\to \Lipb$ is well-defined.\
    Let $t\geq 0$ and $(f_n)_{n \in \N} \subset \Cb$ with $f_n\tomix f$.\ Then, by \eqref{eq:J.contraction}, 
    \begin{equation}\label{eq:proof.conv.mixed1}
    \sup_{n\in \N}\|K(t)f_n\|_\infty<\infty.
    \end{equation} 
    Let $\eps>0$ and $K\Subset \Rd$.\ Then, by Lemma \ref{lem:J.cont.above}, there exists $K'\Subset \Rd$ such that
    \begin{equation}\label{eq:proof.conv.mixed2}
    \|K(t)f_n-K(t)f\|_{\infty,K}\leq \eps \|f-f_n\|_\infty+\|f-f_n\|_{\infty,K'}.
    \end{equation} 
    Approximating $f \in \Cb$ with a sequence $(f_n)_{n \in \N} \subset \Lipb$, see Lemma \ref{lem:approx.res}, \eqref{eq:proof.conv.mixed1} and \eqref{eq:proof.conv.mixed2} imply that $K(t)f\in \Cb$.\ Moreover, by \eqref{eq:proof.conv.mixed1} and \eqref{eq:proof.conv.mixed2}, it follows that $K(t)\colon \Cb\to \Cb$ is continuous.
    
    Now, assume that $(t_n)_{n \in \N} \subset [0,\infty)$ with $t_n\to 0$ and $(f_n)_{n \in \N} \subset \Cb$ with $f_n\tomix f\in \Cb$.\ Then, for all $n\in \N$ and $x\in \Rd$,
    \[
    \big|\big(K(t_n)f_n\big)(x)-f(x)\big|\leq \sup_{\pi\in \operatorname{P}_{t_n}} \big|\big(\mathcal J(\pi)f_n\big)(x)-f(x)\big|,
    \]
    and the claim follows from Lemma \ref{lem:right.continuity}.
\end{proof}

\section{Proofs of results in Section \ref{sec:DRO}}\label{sec:proofs.single.period}

\begin{proof}[Proof of Proposition \ref{lem:I.on.Cb}]
    We start by showing that $I(t)$ is well defined as an operator $\Lipb \to \Lipb$ for all $t \ge 0$, and we later extend such operator to $\Cb$ using Lemma \ref{lem:J.cont.above}. To that end, let $t\geq 0$ and $f\in\Lipb$.\ Then, by \eqref{eq.lipschitz.apriori}, for every $x_1,x_2\in\X$,
    \begin{align*}
        \big| \big(I(t)f\big)(x_1) - \big(I(t)f\big)(x_2) \big| & \le \sup_{a \in A} \sup_{\nu \in {\mathcal{B}}_t^a(m)} \int_\X \big|f\big(\psi_t^a(x_1) + z\big) - f\big(\psi_t^a(x_2) + z\big)\big|\,\nu(\d z) \\
        & \le e^{ct} \|f\|_\Lip \|x_1 - x_2\|,
    \end{align*}
     so that $I(t)f \in \Lipb$ with $\|I(t)f\|_\Lip \le e^{ct}\|f\|_\Lip$ and (ii) is proved.
    
    We can then use Lemma \ref{lem:J.cont.above} to extend $I(t)$ to a map $\Cb \to \Cb$ for all $t \ge 0$, since, for $f \in \Cb$, $(f_n)_{n \in \N} \subset \Lipb$ with $f_n \tomix f$, see Lemma \ref{lem:approx.res}, and $K \Subset \X$, it holds
    \[
    \|I(t)f - I(t)f_n\|_{\infty, K} \le \big\|J(t)|f - f_n|\big\|_{\infty, K} \to 0 \quad \text{as } n \to \infty.
    \]
    Therefore, $I(t)f$ is continuous, being a uniform limit on compacts of Lipschitz continuous functions. The monotonicity of $I(t)$ follows directly from the definition and, in order to show that it is a contraction, it is enough to write
    \[
    |(I(t)f)(x) - (I(t)g)(x)| \le \sup_{a \in A} \sup_{\nu \in {\mathcal{B}}_t^a(m)} \int_\X \big|f\big(\psi_t^a(x) + z\big) - g\big(\psi_t^a(x) + z\big)\big|\,\nu(\d z) \le \|f - g\|_\infty.
    \]
    Therefore, (i) is proved.
    
    We now prove (iii).\ Let $s,t \ge 0$ and $f \in \Cb$ be given.\ Note that, given $a \in A$, the operator $I^a$, defined in \eqref{eq:def.Ia}, corresponds to a specific instance of the operator $I$ studied in \cite{nendel2021wasserstein} (obtained by choosing $(\mu_t^a)_{t \ge 0}$ as reference family of measures and $\phi(v) = \infty \cdot \one_{(m, \infty)}(v)$, for $v \ge 0$, as penalty term). Moreover, Assumption~\ref{ass:general} is stronger than \cite[Assumption 2.2]{nendel2021wasserstein}. Therefore, we can use \cite[Lemma~3.10]{nendel2021wasserstein} to get that, for every $x\in \X$,
    \[
    \big(I^a(t)I^a(s)f\big)(x) \leq \big(I^a(t + s)f\big)(x).
    \]	
    This ensures that, for every $x\in \X$,
    \begin{align*}
        (I(t + s)f)(x) = \big( \inf_{a \in A} I^a(t + s)f \big)(x) &\ge \big(\inf_{a \in A} I^a(t)I^a(s)f \big)(x) \\
        & \ge \big( \inf_{a \in A} I^a(t) \inf_{a' \in A} I^{a'}(s)f \big)(x)\\ 
        &= (I(t)I(s)f)(x),
    \end{align*}
    as claimed.
\end{proof}

 \begin{proof}[Proof of Proposition \ref{lem:generator}]
    The proof closely follows the arguments in the proof of \cite[Lemma~3.11]{nendel2021wasserstein}. Let $f\in \Cb^1$. First, recall the definition of $I^a$ from \eqref{eq:def.Ia} and observe that
    \[
    \frac{\big(I(t)f\big)(x) - \big(T(t)f\big)(x)}{t}\leq \sup_{a\in A}\frac{\big(I^a(t)f\big)(x) - \big(T^a(t)f\big)(x)}{t}.
    \]
    Using \eqref{eq.le.gen.1}, we obtain, for all $t \ge 0$ and $x \in \X$,
    \[
    \sup_{a\in A}\frac{\big(I^a(t)f\big)(x) - \big(T^a(t)f\big)(x)}{t} \le m \|\nabla f\|_\infty.
    \]
    Now, let $K \Subset \X$ and $\eps>0$. We prove that, for $t > 0$ sufficiently small,
    \begin{equation}\label{eq.le.gen.2}
    \sup_{x \in K} \bigg( \sup_{a\in A}\frac{\big(I^a(t)f\big)(x) - \big(T^a(t)f\big)(x)}{t} - m\|\nabla f(x)\| \bigg) \le \eps.
    \end{equation}
    Using the Fundamental Theorem of Calculus and $\gamma_t^{a,x}\in \cP_p(\R^d\times \R^d)$ as in Remark \ref{rem.gamma}, we estimate
    \begin{align*}
        &\sup_{a\in A}\frac{\big(I^a(t)f\big)(x) - \big(T^a(t)f\big)(x)}{t} = \sup_{a \in A} \frac{1}{t} \int_{\X \times \X} \big(f(\psi_t^a(x) + z) - f(\psi_t^a(x) + y)\big)\,\gamma_t^{a,x}(\d y, \d z) \nonumber \\
        &\quad = \sup_{a \in A} \frac{1}{t}\int_0^1\int_{\X \times \X} \nabla f\big(\psi_t^a(x)+y+s(z-y)\big)^\top(z-y)\,\gamma_t^{a,x}(\d y, \d z) \,\d s\\
        &\quad \leq m\|\nabla f(x)\|+\sup_{a \in A} \frac{1}{t}\int_0^1\int_{\X \times \X} \big\|\nabla f\big(\psi_t^a(x)+y+s(z-y)\big)-\nabla f(x)\big\| \|z-y\|\,\gamma_t^{a,x}(\d y, \d z)\,\d s,
    \end{align*}
     and it remains to bound the last term by $\eps$ uniformly for $x\in K$.\ Using the continuity of $\nabla f$ together with Assumption \ref{ass:general} (i), there exists some $\delta>0$ such that
    \begin{align*}
    \sup_{x\in K}\sup_{a \in A}& \frac{1}{t}\int_0^1\int_{\X \times \X} \big\|\nabla f\big(\psi_t^a(x)+y+s(z-y)\big)-\nabla f(x)\big\|\|z-y\|\,\gamma_t^{a,x}(\d y, \d z)\,\d s\\
    &\leq \frac{\eps}2+ 2\|\nabla f\|_\infty\sup_{x\in K}\sup_{a \in A}\frac1{t}\int_{\X \times \X} \|z-y\|\big(\one_{\{\|y-z\|> \delta\}}+\one_{\{\|y\|> \delta\}}\big)\,\gamma_t^{a,x}(\d y, \d z)\\
    &\leq \frac{\eps}2+2\|\nabla f\|_\infty\sup_{x\in K}\sup_{a \in A}\Bigg(\frac{1}{\delta^{p-1}}\frac{\cW_p(\mu_t^a,\nu_t^{a,x})^p}{t}+\bigg(\frac1{\delta}\int_{\X}\|y\|\,\mu_t^a(\d y)\bigg)^{\frac{1}{q}}\frac{\cW_p(\mu_t^a,\nu_t^{a,x})}{t}\Bigg)\\
    &\leq \frac{\eps}2+2m\|\nabla f\|_\infty\Bigg(\frac{t^{p-1}}{\delta^{p-1}}+\sup_{a\in A}\bigg(\frac1{\delta}\int_{\X}\|y\|\,\mu_t^a(\d y)\bigg)^{\frac1q}\Bigg)\leq \eps,
    \end{align*}
    where, in the second step, we used Markov's inequality twice.\ Moreover, using the observation
    \begin{align*}
    \frac{(I(t)f)(x) - (T(t)f)(x)}{t}&\geq \inf_{a\in A} \frac{\big(I^{a}(t)f\big)(x)-T^a(t)f\big)(x)}{t}\\
    &\geq -\sup_{a\in A} \frac{\big(I^a(t)(-f)\big)(x) - \big(T^a(t)(-f)\big)(x)}{t},
    \end{align*}
    the claim follows from \eqref{eq.le.gen.1} and \eqref{eq.le.gen.2}. 
    \end{proof}

\section{Proofs of results in Section \ref{sec:multi.DRO}}\label{sec:proofs.DRO.multi}

In the proof of Theorem \ref{thm:semigroup}, we work with an auxiliary operator $\big(\hat{\cS}(t)\big)_{t \ge 0}$, which we define by using \textit{dyadic} partitions. To that end, we denote by $\cD$ the set of dyadic numbers, i.e.,
	\begin{equation} \label{eq:S.dyadic.D}
		\cD:=\big\{k 2^{-n} \colon k,n \in \N\cup\{0\} \big\}.
	\end{equation}
    Next, we define
    \begin{align} \label{eq:S.dyadic.0}
		\pi_t^n:=\{0, 2^{-n},2 \cdot 2^{-n},\dots, k_t^n 2^{-n},t\}\quad \text{for } t > 0 \text{ and } n \in \mathbb{N},
	\end{align}
    where $k_t^n := \max\{k \in \mathbb{N} \cup \{0\} \colon k 2^{-n} < t\}$ for $t > 0$ and $n\in \mathbb{N}$, and we use the convention $\pi_0^n := \{0\}$ for all $n \in \N$.
    We will refer to $(\pi_t^n)_{n\in\mathbb{N}}$ as the dyadic partitions of $[0,t]$ for $t\geq0$.

    We can then define the operator $\hat{\mathcal{S}}:=\big(\hat{\mathcal{S}}(t)\big)_{t\geq 0}$
    \begin{equation} \label{eq:S.dyadic}
		(\hat{\cS}(t)f)(x) := \inf_{n \in \N} (\mathcal{I}(\pi_t^n)f)(x) = \lim_{n \to \infty} (\mathcal{I}(\pi_t^n)f)(x), \quad \text{for } t \ge 0,\; f \in \Cb, \text{ and } x \in \X.
	\end{equation}
    where the infimum is well-defined since $((\mathcal{I}(\pi_t^n)f)(x))_{n \in \N}$ is a decreasing sequence which is bounded from below by $-\|f\|_\infty>-\infty$ and the equality between the infimum and the limit follows from the monotonicity of $\cI$ over refining partitions, see Proposition~\ref{lem:I.on.Cb}\;(iii).\ As a consequence, we shall write, for all $n_0 \in \N$,
    \begin{equation} \label{eq:S.dyadic.inf}
        (\hat{\cS}(t)f)(x) = \inf_{n \geq n_0} (\mathcal{I}(\pi_t^n)f)(x) \quad \text{for } t \ge 0,\; f \in \Cb, \text{ and } x \in \X.
    \end{equation}

    With the definition of the previous auxiliary operator at hand, we can state the proof of Theorem \ref{thm:semigroup}.

\begin{proof}[Proof of Theorem \ref{thm:semigroup}]
    We organize the proof in six steps.

    \vspace{0.5em}
    \noindent \textit{Step 1:\ For all $t \ge 0$, $\hat{\cS}(t)$ is a map $\Lipb \to \Lipb$.}\ 
    Iterating the estimate in Proposition \ref{lem:I.on.Cb}~(ii), we get $\|\cI(\pi_n^t)f\|_\Lip \le e^{ct}\|f\|_\Lip$ for all $n \in \N$ and $t \ge 0$, which in turn implies
    \begin{align*}
    \big|\big(\hat{\cS}(t)f\big)(x) - \big(\hat{\cS}(t)f\big)(y)\big| &\le \sup_{n \in \N} |(\cI(\pi_t^n)f)(x) - (\cI(\pi_t^n)f)(y)|\\
    &\le e^{ct}\|f\|_\Lip \|x - y\|
    \end{align*}
    for all $x, y \in \X$.

    \vspace{0.5em}
    \noindent \textit{Step 2: For all $t \ge 0$, $\hat{\cS}(t)$ is a sequentially continuous map $\Cb \to \Cb$ (w.r.t.\ the mixed topology) and $\cI(\pi_t^n)f \tomix \hat{\cS}(t)f$ as $n \to \infty$}. Let $f \in \Cb$ and $(f_k)_{k \in \N} \subset \Lipb$ such that $f_k \tomix f$ as $k \to \infty$ (see Lemma \ref{lem:approx.res}). For all $K \Subset \X$ and $\eps > 0$, using Lemma \ref{lem:J.cont.above}, we obtain
    \begin{equation*}
    \big\|\hat{\cS}(t)f_k - \hat{\cS}(t)f\big\|_{\infty, K} \le \sup_{n \ge n_0} \| \cJ(\pi_t^n)f_k - \cJ(\pi_t^n)f\|_{\infty, K} \le \eps
    \end{equation*}
    for $n_0, k \in \N$ sufficiently large.
    This implies
    \begin{equation} \label{eq:S.hat.cont.lip}
    \big\|\hat{\cS}(t)f_k - \hat{\cS}(t)f\big\|_{\infty, K} \to 0 \quad \text{as } k \to \infty.
    \end{equation}
    Therefore, $\hat{\cS}(t)f \in \Cb$ for all $t \ge 0$ and $f \in \Cb$, being uniform limit over compacts of (Lipschitz) continuous functions.\ 
    The sequential continuity of $\hat{\cS}(t)$ is again obtained from equation \eqref{eq:S.hat.cont.lip} since, for all $(f_k)_{k \in \N} \subset \Cb$, $f \in \Cb$ with $f_k \tomix f$ as $k \to \infty$, one has
    \begin{equation} \label{eq:S.hat.cont}
        \lim_{k \to \infty} \hat{\cS}(t)f_k = \hat{\cS}(t)f \quad \text{for all } t \ge 0.
    \end{equation}
    Finally, using that the sequence $(\cI(\pi_t^n)f)(x)$ is monotonically decreasing for all $x \in \X$ and $\hat{\cS}(t)f \in \Cb$, Dini's Theorem implies that the convergence $(\cI(\pi_t^n)f)(x) \to \big(\hat{\cS}(t)f\big)(x)$ is uniform over compacts, i.e.,
    \begin{equation} \label{eq:dini.hatS}
    \sup_{x \in K} \big|(\cI(\pi_t^n)f)(x) - \big(\hat{\cS}(t)f\big)(x)\big| \to 0 \text{ as } n \to \infty,
    \end{equation}
    for all $f \in \Cb$, $t \ge 0$, and $K \Subset \X$.\ Since $\|\cI(\pi_t^n)f\|_\infty \le \|f\|_\infty$ for all $n \in \N$, we conclude that $\cI(\pi_t^n)f \tomix \hat{\cS}(t)f$ as $n \to \infty$.
    
    \vspace{0.5em}
    \noindent \textit{Step 3: $\hat{\cS}$ is left-continuous in time}.
    We now show that, for every $f \in \Cb$, $t > 0$, and $K \Subset \X$ it holds
		\begin{equation} \label{eq:left.cont}
		\sup_{x \in K} \big|\big(\hat{\mathcal{S}}(s)f\big)(x) - \big(\hat{\mathcal{S}}(t)f\big)(x)\big| \to 0 \qquad \mbox{as } s\uparrow t.
		\end{equation}
    Fix $f \in \Cb$, $t > 0$, $K \Subset \X$, and $\eps > 0$. By \eqref{eq:dini.hatS}, there exists $n_t \in \N$ such that, for every $x \in K$,
    \[
    (\hat{\mathcal{S}}(t)f)(x) \ge (\mathcal{I}(\pi_t^{n_t})f)(x) - \eps/3.
    \]
    Let $s>0$ be such that $t - s < 2^{-n_t}$. Then, for every $x \in K$, it holds
    \begin{align}
        (\hat{\mathcal{S}}(t)f)(x) - (\hat{\mathcal{S}}(s)f)(x) & \ge (\mathcal{I}(\pi_{t}^{n_t})f)(x)-(\mathcal{I}(\pi_s^{n_t})f)(x)  - \eps/3 \nonumber \\
        & \ge \big( \mathcal{I}(\pi_s^{n_t}) I(t-s)f \big)(x)-(\mathcal{I}(\pi_s^{n_t})f)(x)  - \eps/3 \nonumber \\
        & \ge -\big| \big(\mathcal{J}(\pi^{n_t}_s) (I(t-s)f - f)\big)(x) \big| - \eps/3, \label{eq:pr.S.cont.t.0}
    \end{align}
    where the second inequality follows from Proposition \ref{lem:I.on.Cb}\;(iii).

    Since $\|I(t-s)f\|_\infty \le \|f\|_\infty$, we can apply Lemma \ref{lem:J.cont.above} (if necessary, taking $n_t$ larger than previously stated) to obtain that there exist some $K' \Subset \X$ such that
    \begin{align} \label{eq:pr.S.cont.t}
        \|\mathcal{J}(\pi^{n_t}_s)(I(t-s)f - f)\|_{\infty, K} \le  \|I(t-s)f - f\|_{\infty, K'} + \eps/3 \quad \text{for all } s \in [0,t].
    \end{align}
    
    Then by Lemma \ref{lem:right.continuity}, we can choose $\delta \in (0, 2^{-n_t})$ such that $\|I(t-s)f - f \|_{\infty, K'} \le \eps/3$ whenever $s>0$ satisfies $t-s < \delta$. Therefore, combining \eqref{eq:pr.S.cont.t.0} with \eqref{eq:pr.S.cont.t}, we find that, for all $s>0$ with $t-s < \delta$ and $x\in K$,
    \begin{align*}
        (\hat{\mathcal{S}}(t)f)(x)- (\hat{\mathcal{S}}(s)f)(x) & \ge - \|\mathcal{J}(\pi^{n_t}_s)(I(t-s)f - f)\|_{\infty, K} - \eps/3 \\
        & \ge - \|I(t-s)f - f\|_{\infty, K'} - \frac{2}{3}\eps \geq - \eps.
    \end{align*}
    In order to obtain the opposite inequality, we proceed in a similar way, starting from the set $[0,t) \cap  \cD$, where $\cD$ is the set of dyadic numbers, see \eqref{eq:S.dyadic.D}.
    For  every $s \in [0,t) \cap  \cD$, let $n_s \in \N$ be such that, for every $x \in K$,
    \[
    (\hat{\mathcal{S}}(s)f)(x)\geq (\mathcal{I}(\pi_s^{n_s})f)(x) - \eps/3,
    \]
    and we repeat the same estimates as above to obtain some (sufficiently small) $\delta>0$ such that, for every $s \in (t - \delta, t) \cap \cD$ and every $x \in K$,
    \begin{align*}
        (\hat{\mathcal{S}}(t)f)(x) - (\hat{\mathcal{S}}(s)f)(x) & \le (\mathcal{I}(\pi_s^{n_s})I(t-s)f)(x) - (\mathcal{I}(\pi_s^{n_s})f)(x) + \eps/3 \\
        & \le \|\mathcal{J}(\pi_s^{n_s}) (I(t-s)f -f)\|_{\infty, K} + \eps/3 \\
        & \le \|I(t-s)f - f\|_{\infty, K'} + \frac{2}{3}\eps \le \eps.
    \end{align*}
      Finally, to extend the last inequality to every $s$ sufficiently close to $t$, we can choose $\delta_t > 0$ such that, for every  $u \in (t - \delta_t, t) \cap  \cD$,
    \[
    \|\hat{\mathcal{S}}(t)f - \hat{\mathcal{S}}(u)f\|_{\infty,K} \le \eps/2.
    \]
    Analogously,  for every $s \in (t - \delta_t, t)$, we can choose $\delta_s > 0$ such that, for every  $u \in (s - \delta_s, s) \cap  \cD$,
    \[
    \|\hat{\mathcal{S}}(s)f - \hat{\mathcal{S}}(u)f\|_{\infty,K} \le \eps/2.
    \]
    Hence, choosing $\delta = \delta_t$ and then taking any $u \in (t - \delta, t) \cap  (s - \delta_s, s) \cap  \cD$, for each $s \in (t - \delta, t)$, we have that
    \begin{align*}
        \|\hat{\mathcal{S}}(t)f - \hat{\mathcal{S}}(s)f\|_{\infty,K} \le \|\hat{\mathcal{S}}(t)f - \hat{\mathcal{S}}(u)f\|_{\infty,K}+ \|\hat{\mathcal{S}}(s)f - \hat{\mathcal{S}}(u)f\|_{\infty,K} \le \eps,
    \end{align*}
    and the left-continuity of the map $t \mapsto \hat{\cS}(t)f$ follows.

    \vspace{0.5em}
    \noindent \textit{Step 4: Semigroup property of $\hat{\cS}$}.
    In order to prove the semigroup property, i.e., condition (iii) in Definition \ref{def:semigroup}, we follow the strategy presented in \cite{nendel2021wasserstein}. Fix $f \in \Cb$ and, at first,
     $t \in \cD$. Then, for every $x \in \X$ and $s\geq 0$,
    \begin{align*}
        \begin{aligned}
    ( \hat{\mathcal{S}}(t + s)f)(x) = \lim_{n \to \infty} (\mathcal{I}( \pi^n_{t + s})f)(x) &= \lim_{n \to \infty} ( \cI(\pi_t^n)\cI(\pi_s^n)f )(x)\\
    &\geq (\hat{\mathcal{S}}(t)\hat{\mathcal{S}}(s)f)(x),
    \end{aligned}
    \end{align*}
    where the second equality holds because $t\in\mathcal{D}$, and the inequality holds by definition of $\hat{\mathcal{S}}$. 
		
    To obtain the opposite inequality, fix $\hat n \in \N$. Then, for every $x \in \X$ and $s\geq 0$,
    \begin{align*}
    ( \hat{\cS}(t + s)f)(x) & = \lim_{n \to \infty} (\mathcal{I}( \pi^n_{t + s})f)(x)= \lim_{n \to \infty} ( \cI(\pi_t^n)\cI(\pi_s^n)f )(x)\\
    &\le \lim_{n \to \infty} ( \cI(\pi_t^n)\cI(\pi_s^{\hat n})f )(x) = ( \hat{\cS}(t)\cI( \pi_s^{\hat n})f )(x),
    \end{align*}
    where the inequality holds since $\pi_s^{n}$ becomes finer than $\pi_s^{\hat{n}}$ as $n \to \infty$ and $\cI$ is monotonically decreasing over refining partitions.
    
    Moreover, by Step 2, we know that $\cI(\pi_s^{\hat n}) f \tomix \hat{\cS}(s)f$ as $\hat n \to \infty$.\
    Therefore, using the continuity in the mixed topology of $\hat{\cS}(t)$, we obtain, for every $x \in \X$ and $s \geq 0$,
    \[
    ( \hat{\cS}(t + s)f)(x)  \le \lim_{\hat n \to \infty} ( \hat{\cS}(t)\cI(\pi_s^{\hat n})f )(x) = ( \hat{\cS}(t) \hat{\cS}(s) f )(x).
    \]
    To extend the semigroup property to arbitrary $t > 0$, we consider an approximating sequence $(t_n)_{n \in \N} \subset \cD$ such that $t_n \uparrow t$ as $n\to \infty$. Then it holds
    \[
    \hat{\cS}(t + s)f = \lim_{n \to \infty} \hat{\cS}(t_n + s)f = \lim_{n \to \infty} \hat{\cS}(t_n) \hat{\cS}(s)f = \hat{\cS}(t) \hat{\cS}(s)f,
    \]
    where the first and last equality follow from the time continuity of $\hat{\cS}$ from the left, and all the limits above are to be understood w.r.t.\;the mixed topology.

    \vspace{0.5em}
    \noindent \textit{Step 5: $\hat \cS$ is strongly continuous.}
    By Step 3, it is enough to show that
    \begin{equation} \label{eq:S.cont.at.0}
        \hat{\cS}(t)f \tomix f \quad \text{as } t \downarrow 0 \text{ for all } f \in \Cb.
    \end{equation}
    As a matter of fact, for $s \downarrow t$, we can then use the semigroup property and the sequential continuity of $\hat{\cS}(t)$, for fixed $t > 0$, to write
    \[
    \hat{\cS}(s)f = \hat{\cS}(t)\hat{\cS}(s - t)f \tomix \hat{\cS}(t) f \quad \text{as } s \downarrow t.
    \]
    We therefore prove \eqref{eq:S.cont.at.0}. 
    Note that, by \eqref{eq:S.dyadic.inf} and Lemma \ref{lem:right.continuity}, for each compact $K \Subset \X$, $f \in \Cb$, and $\eps > 0$, there exists $h > 0$ such that
    \[
    \sup_{x \in K} (\hat{\cS}(t)f)(x) - f(x) \le \sup_{x \in K} \sup_{n \in \N} \big(\cJ(\pi_t^n)f(x) - f(x)\big) \le \eps
    \]
    for all $t \in [0,h]$. On the other hand, again by \eqref{eq:S.dyadic.inf} and Lemma \ref{lem:right.continuity}, there exists $h > 0$ such that
    \[
    \sup_{x \in K} \big( f(x) - (\hat{\cS}(t)f)(x)\big) \le \sup_{x \in K} \sup_{n \in \N} \big( f(x) + \big(\cJ(\pi_t^n)(-f)\big)(x) \big) \le \eps
    \]
    for all $t \in [0, h]$. Since $\eps > 0$ is arbitrarily small, putting the two estimates together we obtain \eqref{eq:S.cont.at.0}.
    
    \vspace{0.5em}
    \noindent \textit{Step 6: $\hat{\cS} = \cS$.}
    We show that $\hat{\cS}(t)f = \cS(t)f$ for all $t \ge 0$ and $f \in \Cb$. From Proposition \ref{lem:I.on.Cb}\;(iii), it follows that, for every $t\geq 0$, $\cS(t)f \le \hat{\cS}(t)f$, since the infimum in the definition of $\cS(t)f$ is taken over more partitions. Moreover, for all $\pi = \{t_0=0, t_1, \dots, t_k = t\} \in \operatorname{P}_t$
    \begin{align*}
    \cI(\pi)f & = I(t_1 - t_0) I(t_2 - t_1) \cdots I(t_k - t_{k-1})f \\
    & \ge \hat{\cS}(t_1 - t_0) \hat{\cS}(t_2 - t_1) \cdots \hat{\cS}(t_k - t_{k-1})f = \hat{\cS}(t)f,
    \end{align*}
    where the last equality comes from the semigroup property of $\hat{\cS}$.
    By taking the infimum over all $\pi \in \operatorname{P}_t$, we get $\cS(t)f \ge \hat{\cS}(t)f$. 
    
    \vspace{0.5em}
    \noindent It is immediate to show that $\cS$ satisfies properties (i) and (ii) in Definition \ref{def:semigroup}. Moreover, we have shown in Step 4 and Step 5, respectively, that properties (iii) and (iv) are satisfied by $\hat{\cS}$ and therefore by $\cS$. Therefore, $\cS$ is a strongly continuous monotone semigroup on $\Cb$.
    To obtain \eqref{eq:equivalent.S}, note that, using \eqref{eq:S.dyadic.inf}, we can restrict the definition of $\hat{\cS}$ to arbitrary fine partitions. Therefore,
    \[
    \cS(t)f\leq \inf_{\substack{\pi \in \operatorname{P}_t\\ \mesh \pi \le h}} \cI(\pi)f \leq \inf_{n\geq n_0} \cI(\pi_t^n)f= \hat{\cS}(t)f \quad \text{for all } h>0,
    \]
    where $n_0\in \N$ with $n_0\geq \frac1h$.\ The proof is complete.
\end{proof}

\begin{proof}[Proof of Theorem \ref{thm:semigroup.generator}] We prove the statement in two steps.

\vspace{0.5em}
\noindent {\it Step 1: Generator of the one-period DRO.} Let $f \in \cap_{a\in A} D(\mathscr{L}^a) \cap \Cb^1$. We start proving that
    \begin{align} \label{eq:cor.gen}
        \lim_{t \downarrow 0} \frac{I(t)f - f}{t} = \inf_{a \in A} \mathscr{L}^a f + m \|\nabla f(\,\cdot\,)\|,
    \end{align}
    where the limit is to be understood w.r.t.\ the mixed topology.
	
    To that end, we first show that $\big\| \frac{I(t)f- f}{t} \big\|_\infty<\infty$ for all $t>0$. Indeed, using \eqref{eq.le.gen.1} we get, for every $t>0$,
    \begin{align*}
        \bigg\| \frac{I(t)f- f}{t} \bigg\|_\infty
        & \le \bigg\| \frac{I(t)f - T(t)f}{t} \bigg\|_\infty +
        \bigg\| \frac{T(t)f - f}{t} \bigg\|_\infty \\
        & \le m \|\nabla f\|_\infty + \sup_{a \in A} \bigg\| \frac{T^a(t)f - f}{t} \bigg\|_\infty.
    \end{align*}
    Since, for each $a \in A$, $T^a$ is a linear semigroup on $\Cb$, we can apply
    \cite[Lemma\;4.1.14]{jacob2001pseudo} to obtain, for every $t>0$,
    \begin{align}
       \notag \sup_{a \in A} \bigg\| \frac{T^a(t)f - f}{t} \bigg\|_\infty  = \sup_{a \in A} \bigg\| \frac{1}{t} \int_0^t T^a(s) \sL^a f\,\d s \bigg\|_\infty &\le \sup_{a \in A} \frac{1}{t} \int_0^t \|T^a(s) \sL^a f\|_\infty\,\d s\\
     \label{eq.estimate-generator}   & \le \sup_{a \in A} \| \sL^a f\|_\infty < \infty,
    \end{align}
    where the second inequality holds because $\|T^a(t) g\|_\infty \le \|g\|_\infty$ for all $a \in A$, $t\geq0$, and $g \in \Cb$, and the last inequality follows from condition \eqref{eq:ass.gen.bound}.\
    Moreover, for every $K\Subset \X$, we write
    \begin{align*}
        & \bigg\| \frac{I(t)f- f}{t} - \bigg( \inf_{a \in A} \mathscr{L}^a f +m\|\nabla f(\,\cdot\,)\| \bigg)\bigg\|_{\infty,K} \\
        &\quad  \le \bigg\| \frac{I(t)f - T(t)f}{t} - m \|\nabla f(\,\cdot\,)\| \bigg\|_{\infty,K} + \bigg\| \frac{T(t)f - f}{t} - \inf_{a \in A} \mathscr{L}^a f \bigg\|_{\infty,K} \\
        &\quad \le \bigg\| \frac{I(t)f - T(t)f}{t} - m \|\nabla f(\,\cdot\,)\| \bigg\|_{\infty,K}  + \sup_{a \in A} \bigg\| \frac{T^a(t)f - f}{t} - \mathscr{L}^a f \bigg\|_{\infty,K},
    \end{align*}
	where the first and the second term in the last line vanish as $t \downarrow 0$ because of Proposition \ref{lem:generator} and condition \eqref{eq:ass.generator}, respectively.
    Therefore, the limit \eqref{eq:cor.gen} holds, as claimed.
		
    \vspace{0.5em}
    \noindent {\it Step 2: Extension to the multi-period DRO.} By Step 1, for any $\eps > 0$, $f \in \cap_{a\in A} D(\mathscr{L}^a) \cap \Cb^1$ and $K\Subset \X$, and any sufficiently small $t > 0$, it holds
    \begin{align*}
        &\sup_{x \in K} \bigg( \frac{(\cS(t)f)(x) - f(x)}{t} - \Big(\inf_{a \in A} \mathscr{L}^a f(x) +m\|\nabla f(x)\| \Big) \bigg) \\
        &\quad  \le \sup_{x \in K} \bigg( \frac{(I(t)f)(x) - f(x)}{t} - \Big(\inf_{a \in A} \mathscr{L}^a f(x) +m\|\nabla f(x)\| \Big) \bigg) < \eps.
    \end{align*}
    To show the other inequality, we write
    \begin{align*}
        \frac{(\cS(t)f)(x) - f(x)}{t} & = \inf_{\pi \in \operatorname{P}_t} \frac{(\cI(\pi)f)(x) - f(x)}{t} \\
        & = \inf_{\pi \in \operatorname{P}_t} \sum_{i = 0}^{k - 1} \frac{\big(\cI(\{t_0, \dots, t_i\})I(t_{i+1} - t_i)f\big)(x) - \big(\cI(\{t_0, \dots, t_i\})f\big)(x)}{t} \\
        & \ge - \sup_{\pi \in \operatorname{P}_t} \sum_{i=0}^{k-1} \frac{t_{i+1} - t_i}{t} \cJ(\{t_0, \dots, t_i\})\bigg( \frac{f - I(t_{i+1} - t_i)f}{t_{i+1} - t_i} \bigg)(x) \\
        & \ge - \sup_{\max \pi \le t} \sup_{h \in (0,t]} \cJ(\pi) \bigg( \frac{f - I(h)f}{h} \bigg)(x),
    \end{align*}
    where, in each step, $\pi=\{t_0,\ldots,t_k\}$ with $0=t_0<\cdots<t_k=t$.\ Therefore, combining \eqref{eq:cor.gen} with Lemma \ref{lem:right.continuity}, we get that, for every $f \in \cap_{a\in A} D(\mathscr{L}^a) \cap \Cb^1$ and sufficiently small $t > 0$,
    \begin{align*}
        & \sup_{x \in K} \bigg( -\frac{(\cS(t)f)(x) - f(x)}{t} + \inf_{a \in A} (\mathscr{L}^a f)(x) + m \|\nabla f(x)\| \bigg) \\
        &\quad  \le \sup_{x \in K} \sup_{\max \pi \le t} \sup_{h \in (0,t]} \bigg( \cJ(\pi)\bigg( \frac{f - I(h)f}{h} \bigg)(x) + \inf_{a \in A} (\mathscr{L}^a f)(x) + m \|\nabla f(x)\| \bigg) < \eps,
    \end{align*}
    which concludes the proof.
\end{proof}
\section{Proofs of results in Section \ref{sec:apply}}\label{sec:proofs.sec.appl}

\begin{proof}[Proof of Theorem \ref{cor:viscosity}]\label{sec:proof:cor:viscosity}
Let the family $K:=(K(t))_{t\geq 0}$ of operators $\Cb\to\Cb$ be given by \eqref{eq.def.K}, see Lemma \ref{lem:Kconvex}.\ Then, by \eqref{eq:equivalent.S},
\begin{align*}
\cS(t)\big(\lambda f+(1-\lambda) g\big)&=\lim_{h\downarrow 0}\inf_{\substack{\pi\in \operatorname{P}_t\\ \mesh\pi\leq h}}\cI(\pi)\big(\lambda f+(1-\lambda) g\big)\leq \lim_{h\downarrow 0}\inf_{\substack{\pi\in \operatorname{P}_t\\ \mesh\pi\leq h}}\big(\lambda \cI(\pi)f+(1-\lambda) \cJ(\pi)g\big)\\
&\leq \lambda \cS(t)f+(1-\lambda)K(t)g
\end{align*}
for all $\lambda\in [0,1]$, $f,g\in \Cb$, and $t\geq 0$.\ Hence, by Theorem \ref{thm:semigroup} and Theorem \ref{thm:semigroup.generator}, it follows that $\cS=(\cS(t))_{t\geq 0}$ is a $K$-convex monotone semigroup with generator $\mathscr L$ and $S(t)0= 0$ for all $t\geq 0$. Moreover, by Lemma \ref{lem:Kconvex}, $K$ is strongly right-continuous in the sense of \cite[Definition~2.7]{fuchs2025existence}, so that we can apply \cite[Theorem~3.1]{fuchs2025existence} with $M:=\{\delta_x\,|\,x\in \X\}$, where $\delta_x$ denotes the Dirac measure with barycenter $x\in \X$, to conclude that the function $v\colon [0,\infty)\times \mathbb{R}^d \to \mathbb{R}$, defined in \eqref{eq:viscosity_semi}, is a $D$-viscosity solution to the PDE \eqref{eq:viscosity}.
\end{proof}

\begin{proof}[Proof of Corollary \ref{cor:g_exp}]\label{proof:apdx:cor:g_exp}
Since $b(a)\equiv b$ is independent of $a\in A$, also $X_t^{x,a}=X_t^x$ is independent of $a\in A$.\ For the sake of a simplified notation, we therefore suppress the control $a\in A$ for the dynamics $(X_t^x)_{t\in [0,T]}$ throughout this proof.\ We start by proving (i).\ To that end, let $f\in \Cb^2$.\ Then, using It\^o's formula, for every $t>0$ and $x\in\X$,
\begin{align}\label{eq:est00}
    \frac{(T^a(t)f)(x)-f(x)}{t}=\frac{1}{t}\mathbb{E}^{\mathbb P}\Big[f(X_t^x)-f(x)\Big]=\frac{1}{t}\mathbb{E}^{\mathbb P}\bigg[\int_0^t (L f)(X_s^x)\, \d s \bigg],
\end{align}
where
\begin{align}\label{eq:candi_gen}
    (L f)(x):=\frac{1}{2}\Delta f(x)+\big\langle b, \nabla f(x)\big\rangle,
\end{align}
independent of $a\in A$. Since $Lf\in \Cb$ and $X_t^x-x=bt+W_t$ for all $x\in \X$ and $a\in A$, part (i) follows.

We now prove (ii). By part (i), Theorem \ref{thm:semigroup.generator}, and Theorem \ref{cor:viscosity}, the function $v$ given in \eqref{eq:viscosity_g_exp_value} is a $\Cb^2$-viscosity solution to \eqref{eq:viscosity_g_exp}.
Moreover, since $h\in \Cb$ and the underlying process $(X_t^x)_{t\in [0,T]}$ is Markovian, the conditional $g$-expectation value function $V$ given in \eqref{eq:value_g_exp} corresponds to the unique solution $V(t,x)=Y_t^{t,x}$ to the following (Markovian) forward-backward SDE: for every $(t,x)\in [0,T]\times \X$,
\begin{align*}
    \begin{aligned}
        &Y_s^{t,x}=h(X_{T}^{t,x})+\int_t^{T}m\|Z_s^{t,x}\|\,\d s-\int_t^{T}Z_s^{t,x}\, \d W_s,\\
        & X_{s}^{t,x}=
        x+ b(s-t) +(W_s-W_t),\quad \mbox{for $s\in[t,T]$},
    \end{aligned}
\end{align*}
with $(Y^{t,x},Z^{t,x})\in \mathbb{S}^2(\mathbb{R})\times \mathbb{L}^2(\mathbb{R}^d)$, see Definition \ref{dfn:g_exp}.

Therefore, we can apply \cite[Theorem 3.4]{barles1997backward} to ensure that $V(t,x)=Y^{t,x}_t$ is a viscosity solution to the PDE~\eqref{eq:viscosity_g_exp} in the classical sense of \cite{barles1997backward}.\ Then, by the comparison principle for viscosity solutions to a general class of nonlinear PDEs that includes \eqref{eq:viscosity_g_exp}, see, e.g., \cite[Proposition 5.5, Section 5]{neufeld2017nonlinear}, \cite[Theorem 3.5]{barles1997backward}, we conclude that the two functions $V$ and $v$ coincide, as claimed. \end{proof}

	\begin{proof}[Proof of Corollary \ref{cor:SDG}]\label{proof:apdx:cor:SDG} Since $b\colon A\to\X$, and $\sigma\colon A\to \mathbb{R}^{d\times d}$ are  bounded, one can use similar arguments as in the proof of Corollary \ref{cor:g_exp}\;(i) in order to prove part (i) and that, for every $f\in \Cb^2$
	\begin{align*}
		(\mathscr{L}f)(x)&=\inf_{a\in A}\mathscr{L}^af(x)+m\|\nabla f(x)\|\\
		&=\inf_{a\in A}\bigg\{\frac{1}{2}\operatorname{tr}\big(\sigma(a)\sigma(a)^\top \nabla^2 f(x) \big) + \big\langle b(a), \nabla f(x) \big\rangle \bigg\} + m\|\nabla f(x)\|,\quad \text{for all }x\in \X.
	\end{align*}
	For part (ii), we apply Theorem \ref{thm:semigroup.generator} and Theorem \ref{cor:viscosity} with $D:=\Cb^2$ to obtain that the function $v$, given in \eqref{eq:viscosity_SDG_value}, is a $D$-viscosity solution to \eqref{eq:HJBI}.
	Moreover, it follows from \cite[Theorem 2.6]{fleming1989existence} that the robust optimization problem~$V$ in \eqref{eq:value_SDG} is the unique viscosity solution to the following HJBI equation:
	\begin{align}\label{eq:HJBI2}
		\left\{
		\begin{aligned}
			&\partial_tV(t,x)+\inf_{a\in A}\sup_{\|\theta\|\leq m} \big\{\mathscr{L}^{a,\theta}V(t,x) \big\}=0\quad&&\text{for all }(t,x)\in [0,T)\times \X,\\ 
			&V(T,x)=h(x) &&\text{for all }x\in \X,
		\end{aligned}
		\right.
	\end{align}     
	where
	\[
	\mathscr{L}^{a,\theta}V(t,x)=\frac{1}{2}\operatorname{tr}\big(\sigma(a)\sigma(a)^\top \nabla^2 V(t,x) \big) + \big\langle b(a)+\theta, \nabla V(t,x)\big\rangle\quad \text{for }a\in A\text{ and }\|\theta\|\leq m.
	\]
	In particular, since $\sup_{\|\theta\|\leq m} \langle \theta,w \rangle = m \|w\|$ for every $w\in \X$, the HJBI equation in \eqref{eq:HJBI2} coincides with the PDE \eqref{eq:HJBI}.\ Therefore, by the uniqueness result from \cite[Theorem 2.6]{fleming1989existence}, the robust optimization problem\;$V$ coincides with the function $v$ given in \eqref{eq:viscosity_SDG_value}, as claimed.
	\end{proof}

\appendix
\section{Supplementary statements}\label{sec:apdx}
\subsection{Approximation of continuous functions}
We often approximate functions in $\Cb$ (in the mixed topology) with functions with stronger regularity, e.g., functions in $\Lipb$ or $\rm{C_c^2}$. In fact any function $f \in \Cb$ can be approximated in the mixed topology by a sequence of functions in $\Cci$. This follows directly from the fact that  $\Rd$ is a $\sigma$-compact space.
\begin{lemma} \label{lem:approx.res}
    Let $f \in \Cb$.\ Then there exists a sequence $(f_k)_{k \in \N} \subset \Cci$ such that $f_k \tomix f$
    as $k \to \infty$ and $\limsup_{k\to \infty} \|f_k\|_\infty\leq \|f\|_\infty$.
\end{lemma}
\begin{proof}
    Let $K_k := \{x\in\X\,|\,\|x\|\leq k\}$ for all $k \in \N_0$.\
    Then, by the Stone-Weierstrass theorem, for every $k \in \N$, there exists a smooth function $g_k\colon \X\to\R$ with $\|g_k - f\|_{\infty, K_k} \leq \frac1k$.\ Moreover, we consider a family of cut-off functions $(\chi_k)_{k\in \N} \subset \Cci$ with $0 \le \chi_k \leq  1$, $\chi_k(x)=1$ for all $x\in K_{k-1}$, and $\chi_k (x) = 0$ for all $x \in \X\setminus K_k$. Then, $f_k:=g_k\chi_k\in \Cci$ with $$\|f_k\|_\infty\leq \|g_k\|_{\infty,K_k}\leq \|f\|_{\infty,K_k}+\frac1k\leq \|f\|_{\infty}+\frac1k$$ for all $k\in \N$ and, for all $\eps>0$ and $K\Subset \X$, there exists some $k_0\in \N$ with $k_0\geq \frac{1}{\eps}$ and $K\subset K_{k_0-1}$, so that
    \[
     \|f_k-f\|_{\infty,K} \leq \|g_k-f\|_{\infty,K_{k-1}}\leq \|g_k-f\|_{\infty,K_k}
     \leq \frac1k\leq \eps\quad \text{for all }k\in \N\text{ with }k\geq k_0.
    \]
 We have therefore shown that $\Cci\ni f_k\tomix f$ as claimed.
\end{proof}
	
\subsection{A priori estimates on the linear semigroups}
In this section, we report an a priori estimate for the linear semigroups $(T^a(t))_{t \ge 0}$, which is central for the proof of Lemma \ref{lem:J.cont.above}.
\begin{lemma}\label{lem:apriori.linsg}
    Suppose that Assumption \ref{ass:general} holds.\ Then, for every $\eps > 0$, there exists a constant $ C'_\eps \geq 0$ such that, for all $f\in {\rm C_c^2}$,
    \[
    \limsup_{t\downarrow 0}\sup_{a \in A} \sup_{x \in \Rd} \bigg| \frac{(T^a(t)f)(x) - f(x)}{t}\bigg| \le \eps \|f\|_\infty + C_\eps' \bigg(\sup_{y \in \X}\big(1+\|y\|\big)\|\nabla f(y)\| + \|\nabla^2 f\|_\infty \bigg).
    \]
\end{lemma}

Before stating the proof of Lemma \ref{lem:apriori.linsg}, we need a simple estimate on the deterministic drift of the controlled dynamics. The following lemma enables us to keep $\psi_t^a(x)$ outside any arbitrarily large neighborhood of~$0$, uniformly in $a \in A$, by taking $x$ out of a larger neighborhood and choosing $t$ small enough.\ The proof is a straightforward modification of \cite[Lemma 5.4 c)]{nendel2024separable} to our setup.

\begin{lemma}\label{lem:psi.stab}
    Suppose that Assumption \ref{ass:general}\;(i),\;(ii) hold. Then, for every $R' > R \ge 0$,
    there exists $t_0 > 0$ such that, for every $x \in \X$ with $\|x\| \ge R'$ and every $t \in [0, t_0]$,
    \begin{align*}
        \inf_{a \in A}\|\psi_t^a(x)\| \ge R.
    \end{align*}
\end{lemma}
    \begin{proof}
    Let $R' > R \ge 0$. Then, we set $t_0 := \frac{R' - R}{C(1 + R')}$ with $C > 0$ appearing in Assumption \ref{ass:general}\;(i). Then, for every $x \in \X$ with $\|x\|\geq R'$ and every $t \in [0,t_0]$,
    \[
    \sup_{a \in A} \|\psi_t^a(x) - x\| \le t_0 C (1 + \|x\|).
    \]
    Hence, by the inverse triangle inequality together with the fact that $R'>R$, we conclude that, for all $x \in \X$ with $\|x\|\geq R'$ and all $t \in [0,t_0]$,
    \begin{align*}
        \inf_{a\in{A}}\|\psi_t^a(x)\|  \ge \|x\| - \sup_{a\in {A}}\|\psi_t^a(x) - x\| &\ge \|x\| -  t_0 C (1 + \|x\|) = (1 + \|x\|) (1 -  t_0 C) - 1\\
        & = \frac{1 + \|x\|}{1 + R'}(1 + R) - 1 \ge R,
    \end{align*}
    as claimed.
\end{proof}

\begin{proof}[Proof of Lemma \ref{lem:apriori.linsg}]
	Let $\eps>0$.\ Then, by Assumption \ref{ass:general}\,(iv), there exists some constant $M > 0$ such that, for every $t > 0$,
	\[
	\sup_{a \in A} \frac{\mu_t^a\big(\{ y \in \X \,|\, \|y\| > M \}\big)}t \leq  \frac{\eps}{2}.
	\]
    Then, for all $t > 0$, $a \in A$, and $f\in \Cb$,
	\begin{align} \label{eq:proof.ref.dyn.1}
		\bigg| \frac{T^a(t)f(x) - f(x)}{t}\bigg| \le \bigg| \int_{\|y\| \le M} \frac{f(\psi_t^a(x) + y) - f(x)}{t}\,\mu_t^a(\d y) \bigg| + {\eps} \|f\|_\infty.
	\end{align}
    Now, let $f\in {\rm C_c^2}$ and $r\geq 0$ such that $\operatorname{supp}(f)\subset \{x\in \X\,|\,\|x\|\leq r\}$.\ By Lemma \ref{lem:psi.stab}, there exists $R > r$ such that $\|\psi_t^a(x) + y\| > r$ for all $x,y \in \X$ with $\|x\| > R$ and $\|y\| \le M$, $a \in A$, and $t > 0$ sufficiently small.\ Hence, 
    \begin{align} \label{eq:proof.ref.dyn.12}
		f(\psi_t^a(x) + y) - f(x) = 0
	\end{align}
    for all $x,y \in \X$ with $\|x\| > R$ and $\|y\| \le M$, $a \in A$, and $t > 0$ sufficiently small, so that it suffices to estimate the first summand on the right-hand side of \eqref{eq:proof.ref.dyn.1} for $a \in A$, $t> 0$ sufficiently small, and $x\in \R^d$ with $\|x\| \le R$.

    Let $a \in A$, $t> 0$ sufficiently small such that \eqref{eq:proof.ref.dyn.12} holds, and $x\in \R^d$ with $\|x\| \le R$.\ We write
    \begin{align}
    \int_{\|y\| \le M} \frac{f( \psi_t^a(x) + y) - f(x)}{t}\,\mu_t^a(\d y) & = \int_{\|y\| \le M} \frac{f( \psi_t^a(x) + y ) - f(x + y)}{t}\,\mu_t^a(\d y) \notag \\
    & \qquad + \int_{\|y\| \le M} \frac{f(x + y) - f(x)}{t}\,\mu_t^a(\d y), \label{eq:proof.ref.dyn.first}
    \end{align}
	and tackle the two terms separately.\ For the first term, applying a Taylor expansion, we get
    \begin{align}
    \int_{\|y\| \le M} & \frac{f( \psi_t^a(x) + y ) - f(x + y)}{t}\,\mu_t^a(\d y) = \int_{\|y\| \le M} \frac{1}{t} \big\langle \nabla f(x + y), \psi_t^a(x) - x \big\rangle \,\mu_t^a(\d y) \notag \\
    & \qquad + \int_{\|y\| \le M} \int_0^1 \frac{1-s}{t} \big\langle \psi_t^a(x) - x, \nabla^2 f\big(x + y + s (\psi_t^a(x) - x) \big) (\psi_t^a(x) - x) \big\rangle\,\d s \, \mu_t^a(\d y) \notag \\
    & \le \int_{\|y\| \le M} \|\nabla f(x + y)\| \frac{\|\psi_t^a(x) - x\|}{t} \,\mu_t^a(\d y) + \big\|\nabla^2 f \big\|_\infty \frac{\|\psi_t^a(x) - x\|^2}{t} \notag \\
    & \le C (1 + \|x\|) \int_{\|y\| \le M} \|\nabla f(x + y)\|\,\mu_t^a(\d y) + t C^2 (1 + R)^2 \big\|\nabla^2 f \big\|_\infty, \label{eq:bound.f(psi + y)}
    \end{align}
    where $C > 0$ is the constant appearing in Assumption \ref{ass:general}~(i).\
    Since $\mu_t^a$ converges to $\delta_0$ in the Wasserstein-$p$ topology, uniformly over $a \in A$, as $t\to 0$ and $f$ has bounded second derivatives, we obtain
    \[
    C (1 + \|x\|) \int_{\|y\| \le M} \|\nabla f(x + y)\|\,\mu_t^a(\d y) \to C (1 + \|x\|) \|\nabla f(x)\| \quad \text{as }t \to 0
    \]
    uniformly over $a\in A$ and $x \in \Rd$ with $\|x\| \le R$. For $t > 0$ small enough, we can therefore bound \eqref{eq:bound.f(psi + y)} by
    \begin{align}
        \int_{\|y\| \le M} \frac{f( \psi_t^a(x) + y ) - f(x + y)}{t}\,\mu_t^a(\d y)
        & \le C \big((1 + \|x\|) \|\nabla f(x)\| + \big\|\nabla^2 f \big\|_\infty \big). \label{eq:est.I}
    \end{align}
    Note that, in the previous inequality, we have also chosen $t$ small enough so that $t C (1 + R)^2 < 1$.\ For the second term in \eqref{eq:proof.ref.dyn.first}, we use a Taylor expansion for $f$ around $x$ to get
	\begin{align*}
		\int_{\|y\| \le M} \frac{f(x + y) - f(x)}{t}\,\mu_t^a(\d y) & = \frac{1}{t} \int_{\|y\| \le M} \langle \nabla f(x), y \rangle \,\mu_t^a(\d y) \\
        & \qquad + \frac{1}{t} \int_{\|y\| \le M} \int_0^1 (1-s) \langle \nabla^2 f(x+ s y)\,y, y \rangle\,\d s\,\mu_t^a(\d y).
	\end{align*}
	Again we tackle the two terms separately.
	First, let $\chi \in \Cci$ such that $0 \le \chi \le 1$, $\chi(0) = 1$, and $\chi(y) = 0$ for $\|y\|> M$. Note that, for all $i\in \{1,\ldots,d\}$, the function $x \mapsto x_i\chi(x)$ is also in $\Cci$.\ We then have some $C_2 \geq 0$ such that
	\begin{align}
			\frac{1}{t} \int_{\|y\| \le M} \langle \nabla f(x), y \rangle \,\mu_t^a(\d y) & \le \|\nabla f\|_\infty \bigg\| \int_{\|y\| \le M} \frac{y \chi(y) + y(1 - \chi(y))}{t}\,\mu_t^a(\d y) \bigg\| \notag \\
			& \le \|\nabla f\|_\infty \bigg( \bigg\| \int_\X \frac{y\chi(y)}{t} \,\mu_t^a(\d y) \bigg\| + M \int_\X \frac{\chi(0) - \chi(y)}{t}\,\mu_t^a(\d y) \bigg) \notag \\
            & \le C_2 \| \nabla f\|_\infty, \label{eq:est.II.1}
	\end{align}
	where the last inequality follows from  Assumption \ref{ass:general}~(iii).\ Moreover, since the function $x \mapsto \|x\|^2 \chi(x)$ is in $\Cci$ as well, 
	we can use the same arguments presented for~\eqref{eq:est.II.1} to have some $C_3 > 0$ satisfying
	\begin{align}
			\frac{1}{t} \int_{\|y\| \le M} \int_0^1 (1-s) & \langle \nabla^2 f(x+ s y)\,y, y \rangle\,\d s\,\mu_t^a(\d y) \notag \\
            & \le \|\nabla^2 f\|_\infty  \int_{\|y\| \le M} \frac{\|y\|^2 \chi(y)+\|y\|^2(1 - \chi(y))}{t}\,\mu_t^a(\d y) \notag \\
			& \le \|\nabla^2 f\|_\infty \bigg( \int_\X \frac{\|y\|^2 \chi(y)}{t}\,\mu_t^a(\d y) + M^2 \int_\X \frac{\chi(0) - \chi(y)}{t}\,\mu_t^a(\d y) \bigg) \notag \\
            & \le C_3 \|\nabla^2 f\|_\infty. \label{eq:est.II.2}
	\end{align}	
	Combining \eqref{eq:est.I}, \eqref{eq:est.II.1}, and \eqref{eq:est.II.2}, there exists a constant $C_\eps' \geq 0$ such that
	\begin{align*}
		\int_{\|y\| \le M} \frac{f( \psi_t^a(x) + y) - f(x)}{t}\,\mu_t^a(\d y) \le C_\eps'\bigg(\sup_{y \in \X}(1 + \|y\|) \|\nabla f(y)\| + \|\nabla^2 f\|_\infty \bigg) + \eps\|f\|_\infty
	\end{align*}
    for all $t > 0$ small enough.\ The proof is complete.
\end{proof}

\subsection{Continuity in the mixed topology for sublinear monotone operators}\label{sec:apdx:conti_above}

The following lemma is a particular version of \cite[Corollary C.4]{BKN23} tailored to subadditive, positive homogeneous, and monotone operators $\Cb\to \mathrm{F}$, where $\mathrm{F}$ denotes the space of all bounded functions $\Rd \to \R$.\
It provides explicit estimates that allow to verify equicontinuity in the mixed topology for families of such operators by means of suitable cut-off functions.

\begin{lemma} \label{lem:crit_above}
    Let $\mathfrak{I}\colon \Cb\to \mathrm{F}$ be a subadditive, positive homogeneous, and monotone operator, $\eps > 0$, and $K \Subset \Rd$.\ Moreover, assume that there exists a function $\zeta \in \Cb$ and $K' \Subset \Rd$ with
    \begin{enumerate}
        \item[(i)] $0 \le \zeta(x) \le 1$ for all $x \in \Rd$,
        \item[(ii)] $\zeta(x)=0$ for all $x \in \Rd\setminus K'$,
        \item[(iii)] $\big(\mathfrak{I}(1-\zeta)\big)(x) \le \eps$ for all $x \in K$.
    \end{enumerate}
    Then, for all $f,g \in \Cb$,
    \[
    \|\mathfrak{I} f - \mathfrak{I} g \|_{\infty, K} \le \|\mathfrak{I}1\|_\infty \|f - g\|_{\infty, K'} + \eps \|f-g\|_\infty.
    \]
\end{lemma}
\begin{proof}
 Let  $f,g \in \Cb$ and $x\in K$.\ Then, using the fact that $\mathfrak{I}$ is subadditive, positive homogeneous, and monotone, and since $\zeta$ is supported on $K'$ with $0\leq \zeta\leq 1$, 
 \begin{align*}
  \big|(\mathfrak{I} f )(x)- (\mathfrak{I} g)(x)\big|&\leq (\mathfrak{I} |f-g| )(x)\leq \big(\mathfrak{I} (|f-g| \zeta)\big)(x)+\Big(\mathfrak{I} \big(|f-g|(1-\zeta) \big)\Big)(x)\\
  &\leq \|f-g\|_{\infty,K'}(\mathfrak{I}\zeta)(x)+\|f-g\|_\infty \big(\mathfrak{I}(1-\zeta)\big)(x)\\
  &\leq \|\mathfrak{I}1\|_\infty \|f - g\|_{\infty, K'} + \eps \|f-g\|_\infty.
 \end{align*}
\end{proof}

We end this section with a technical result borrowed from \cite{nendel2024separable}, which provides the existence of suitable cut-off functions.

\begin{lemma} \label{lem:cut-off-existence}
    For all $\eps > 0$ and $n \in \N$, there exists a function $\zeta \in \Cci$ with $0 \le \zeta \le 1$, $\zeta(x) = 1$ for all $x \in \Rd$ with $\|x\| \le n$, and
    \begin{equation} \label{eq:cut-off}
        \sup_{x \in \Rd} \big( (1 + \|x\|)\|\nabla \zeta(x)\| + \|\nabla^2 \zeta(x)\| \big) \le \eps.
    \end{equation}
\end{lemma}
\begin{proof}
    It is enough to note that the function $x \mapsto 1 + \|x\|$ satisfies the hypothesis of \cite[Lemma A.1]{nendel2024separable}.\ The result is then a particular version of the aforementioned lemma.
\end{proof}

\bibliographystyle{abbrv}

\end{document}